\documentclass[3p,12pt]{elsarticle}

\usepackage[T1]{fontenc}
\usepackage[latin1]{inputenc}

\usepackage{amsfonts}
\usepackage{amsmath}
\usepackage{amssymb}
\usepackage{wasysym}

\usepackage{epsfig}
\usepackage{epic}
\usepackage{eepic}
\usepackage{graphics}
\usepackage{graphicx}
\usepackage{color}
\usepackage{longtable}
\usepackage{subfigure}

\usepackage{pgfplots}

\usepackage{tikz}
\usetikzlibrary{calc,fadings,decorations.pathreplacing}
\usetikzlibrary{matrix,arrows}

\usetikzlibrary{positioning,fit,shapes,external,3d,patterns,spy}
\usetikzlibrary[shapes.arrows]
\usetikzlibrary{pgfplots.colormaps}

\newcommand{\ds}{\displaystyle}

\newcommand{\diag}{\operatorname{diag}}

\newcommand{\spann}{\operatorname{span}}

\newcommand{\Emb}{\operatorname{E}_\gamma}

\newcommand{\Nn}{{\mathbb N}}

\newcommand{\Rr}{{\mathbb R}}

\newcommand{\Zz}{{\mathbb Z}}

\newcommand{\eps}{\varepsilon}

\newcommand{\Ttt}{\mathbf{T}}
\newcommand{\Ccc}{\mathbf{C}}
\newcommand{\Ddd}{\mathbf{D}}
\newcommand{\Mmm}{\mathbf{M}}

\newcommand{\Pin}{\Pi_n^2}
\newcommand{\Pibasic}{\Pi_{n,p}^2}

\newcommand{\PiL}{\Pi_{n,p}^{2,L}}
\newcommand{\Lagpol}{\mathcal{L}_{n,p} f}

\newcommand{\Lub}{\mathrm{LD}_{n,p}}
\newcommand{\Lubblack}{\mathrm{LD}_{n,p}^{\mathrm{b}}}
\newcommand{\Lubwhite}{\mathrm{LD}_{n,p}^{\mathrm{w}}}
\newcommand{\Lubint}{\mathrm{LD}_{n,p}^{\mathrm{int}}}
\newcommand{\Lubout}{\mathrm{LD}_{n,p}^{\mathrm{out}}}
\newcommand{\Lubvert}{\mathrm{LD}_{n,p}^{\mathrm{vert}}}
\newcommand{\Lubedge}{\mathrm{LD}_{n,p}^{\mathrm{edge}}}
\newcommand{\Lubind}{\Gamma_{n,p}^{L}} 
\newcommand{\Lubindbasic}{\Gamma_{n,p}} 
\newcommand{\LubAbr}{\mathrm{LD}}

\newcommand{\Liscurve}{\gamma_{n,p}} 

\newcommand{\Aa}{\mathcal{A}}
\newcommand{\Bb}{\mathcal{B}}

\newcommand{\equi}{\overset{\Lub}{\sim}}
\newcommand{\Tripi}{\Pi_{n(n+p)}^{\mathrm{e},L}}
\newcommand{\Tripibasica}{\Pi_{n(n+p)}^{\mathrm{e}}}
\newcommand{\Tripibasicb}{\Pi_{n(n+p)-1}^{\mathrm{e}}}
\newcommand{\Tripibasicc}{\Pi_{2n(n+p)-1}^{\mathrm{e}}}

\newcommand{\Lagbastrig}{ e_{i,j} }
\newcommand{\LagbastrigE}{\Emb \! \left( \! \hat{T}_{i}(x) \hat{T}_{j}(y)\! \right)\! }

\newcommand{\1}{\mathrm{1\hspace{-1mm}l}}
\newcommand{\Dx}[1]{\mathrm{d}#1}

\newcommand{\minus}{%
  \setbox0=\hbox{-}%
  \vcenter{%
    \hrule width\wd0 height \the\fontdimen8\textfont3%
  }%
}

\newtheorem{theorem}{Theorem}
\newtheorem{lemma}[theorem]{Lemma}
\newtheorem{proposition}[theorem]{Proposition}
\newtheorem{corollary}[theorem]{Corollary}
\newdefinition{remark}{Remark}
\newproof{proof}{Proof}

\begin{document}

\begin{frontmatter}
\title{Bivariate Lagrange interpolation at the node points of Lissajous curves - the degenerate case}

\author{Wolfgang Erb}

\address{ Institut f\"ur Mathematik \\ 
          Universit\"at zu L\"ubeck \\
	      Ratzeburger Allee 160\\
	      23562 L\"ubeck }
	      
\ead{erb@math.uni-luebeck.de}

\date{\today}

\begin{abstract}
In this article, we study bivariate polynomial interpolation on the node points of degenerate Lissajous figures. These node points form Chebyshev lattices of rank $1$ and are generalizations of the well-known Padua points. We show that these node points allow unique interpolation in appropriately defined spaces of polynomials and give explicit formulas for the Lagrange basis polynomials. Further, we prove mean and uniform convergence of the interpolating schemes. For the uniform convergence the growth of the Lebesgue constant has to be taken into consideration. It turns out that this growth is of logarithmic nature.  
\end{abstract}


\end{frontmatter}

\section{Introduction}

A by now well-established point set for Lagrange interpolation on $[-1,1]^2$ is given by the Padua points~\cite{CaliariDeMarchiVianello2005}. This 
set of points allows unique interpolation in the space $\Pin$ of bivariate polynomials of degree $n$ and has a few outstanding properties: it can be 
characterized as a set of node points of a particular Lissajous curve~\cite{BosDeMarchiVianelloXu2006}, as an affine variety of a polynomial 
ideal~\cite{BosDeMarchiVianelloXu2007} and as a particular Chebyshev lattice of rank $1$~\cite{CoolsPoppe2011}. 
Moreover, the Lagrange interpolant can be computed in a fast and efficient way and the asymptotic of the corresponding Lebesgue constant is of 
log-squared type~\cite{CaliariDeMarchiSommarivaVianello2011,CaliariDeMarchiVianello2008}.

Using affine mappings of the square $[-1,1]^2$, the Padua points can also be used for interpolation on general rectangular domains \cite{CaliarideMarchiVianello2008-2}. 
However, for highly anisotropic rectangles the particular structure of the Padua points is not so well adapted. 
In this case it is more favorable to use interpolation nodes that reflect different resolutions along the axes of the rectangle.
To obtain more flexibility for anisotropic domains, it is therefore reasonable to consider generalizations of the Padua points.

In ~\cite{CoolsPoppe2011,PoppeCools2013}, a framework for approximation and interpolation on general Chebyshev lattices was developed. This framework contains multidimensional anisotropic lattices and the Padua points are included as a special case. But, in this framework the interpolation of functions is in general only available in an approximative way. This approximate interpolation is known as hyperinterpolation~\cite{Sloan1995} and is frequently used in applications. For trivariate Lissajous curves, it is studied in the recent work \cite{BosDeMarchiVianello2015}. 

A different way of generalizing the theory of the Padua points was recently presented in \cite{ErbKaethnerAhlborgBuzug2014}. In this article, the node points of non-degenerate Lissajous curves were used as interpolation points. These nodes turned out to be special Chebyshev lattices of rank $1$. In particular, in this theory also anisotropic point sets can be chosen. However, since the generating curve of the Padua points is a degenerate Lissajous curve, the Padua points are not directly included in the theory developed in \cite{ErbKaethnerAhlborgBuzug2014}.

The goal of this article is to develop an interpolation theory for node points of degenerate Lissajous curves that contains the Padua points as a special case. To this end, we rebuild the interpolation theory developed in \cite{ErbKaethnerAhlborgBuzug2014} for degenerate Lissajous curves. While most of the results carry over, some technical aspects in the proofs differ considerably. This is mainly due to the different geometric properties of the underlying curves and its interpolation nodes. 

The search for favorable node points in multivariate polynomial interpolation has a long-standing history. We refer to the survey articles \cite{GascaSauer2000b,GascaSauer2000} for a general overview. Beside the Padua points, the points introduced by Morrow and Patterson \cite{MorrowPatterson1978} and Xu \cite{Xu1996} (see also \cite{Harris2010}) are of special importance for the theory presented in this article.  

We start our research by studying the node points $\Lub$ of degenerate Lissajous curves. It turns out that these node points can be characterized in several ways and, in particular, as Chebyshev lattices of rank $1$. As in the non-degenerate case, the node points $\Lub$ can be used as quadrature rules on $[-1,1]^2$ for integrals with a product Chebyshev weight. Compared to non-degenerate Lissajous curves, the node sets $\Lub$ in the degenerate setting are smaller, contain more asymmetries and include two vertices of the square $[-1,1]^2$. 

The main results of this article can be found in Section $4$ and $5$. In Theorem \ref{thm:interpolation problem}, we prove that $\Lub$ allows unique polynomial interpolation in a properly defined space $\PiL$ of bivariate polynomials. We will derive an explicit formula 
for the corresponding fundamental Lagrange polynomials. This explicit formula is very similar to the one known for the Padua points and allows to compute the interpolating polynomial in a simple and efficient way.

Whereas in \cite{ErbKaethnerAhlborgBuzug2014} stability and convergence was studied only 
numerically, in this article we investigate the convergence of the interpolation scheme also in an analytic way. 
For continuous functions $f$ we show mean convergence of the Lagrange interpolant in the $L^r$-norm.
For the Xu and the Padua points it is known that the Lebesgue constants grow as $\mathcal{O}(\ln^2 n)$ 
(cf. \cite{BosDeMarchiVianello2006,BosDeMarchiVianelloXu2007,VecchiaMastroianniVertesi2009}). 
We will confirm a similar log-squared behavior also for the Lebesgue constant of the general interpolation scheme considered in this article. We conclude
this article with some numerical experiments that confirm the convergence results and illuminate the role of the parameter $p$ in anisotropic setups.

\section{The node points of degenerate Lissajous curves} \label{sec:Lissajous}

In this work, we consider Lissajous figures of the type
\begin{equation} \label{eq:generatingcurve}
 \Liscurve: \Rr \to [-1,1]^2, \quad \Liscurve(t) = \Big( \cos(nt), \, \cos((n+p)t) \Big),
\end{equation}
with positive integers $n$ and $p$ such that $n$ and $n+p$ are relatively prime. The curve $\Liscurve$ is $2\pi$-periodic but doubly traversed as $t$ varies from $0$ to $2\pi$. For this reason, $\Liscurve$ is referred to as degenerate Lissajous curve and we can restrict the parametrization of the curve to the interval $[0,\pi]$. The points $\Liscurve(0) = (1,1)$ and $\Liscurve(\pi) = ((-1)^n, (-1)^{n+p})$ denote the starting and the end point of the curve $\Liscurve$, respectively. A classical reference for the characterization of two-dimensional Lissajous curves and its singularities is the dissertation \cite{Braun1875} of Braun. In recent years, Lissajous curves are particularly studied in terms of knot theory \cite{BogleHearstJonesStoilov1994,KoseleffPecker2011}.

If we sample the curve $\Liscurve$ along the $n(n+p)+1$ equidistant points 
\[ t_k := \frac{\pi k}{ n (n+p)}, \quad k = 0, \ldots, n(n+p),\]
in the interval $[0,\pi]$, we get the following set of node points:
\begin{equation}
\Lub := \Big\{ \Liscurve (t_k): \quad k = 0, \ldots, n(n+p) \Big\}. \label{def:Luebeckpoints}
\end{equation}

To derive particular properties of the Lissajous curve $\Liscurve$ and the set $\Lub$ it is easier 
to characterize the curve $\Liscurve$ with help of the Chebyshev polynomials $T_n(x) = \cos (n \arccos x)$ of the
first kind. Based on this relation, the following properties of $\Liscurve$ are proven in \cite{Fischer}[Section 3.9] and \cite{KoseleffPecker2011}. 
We use the notation 
\begin{equation*}
	z_k^{n} := \cos\left(\frac{k\pi}{n}\right), \quad n \in \Nn, \; k = 0, \ldots, n,
\end{equation*}
to abbreviate the Chebyshev-Gau{\ss}-Lobatto points.

\begin{proposition} \label{prop-1}
If $n$ and $n+p$ are relatively prime, the 
Lissajous curve $\Liscurve$, $t \in [0,\pi]$, corresponds to the (plane) algebraic curve 
\begin{equation} \mathcal{C}_{n,p} := \{(x,y) \in [-1,1]^2:\;T_{n+p}(x) - T_n(y) = 0\}. \label{eq:chebycurve}
\end{equation} 
The curve $\Liscurve(t)$, $t \in [0,\pi]$, has $\frac{(n+p-1)(n-1)}{2}$ ordinary self-intersection points in the interior of the square $[-1,1]^2$. They are given as
\begin{equation}
\Lubint : = \left\{  \Big( z_{i}^{n+p}, z_{j}^{n}\Big): \quad \begin{array}{l} i = 1, \ldots, n+p-1 \\ j = 1, \ldots, n-1 \\ i+j = 0 \mod 2  \end{array}
\right\} \label{eq:LDinterior} 
\end{equation}
and can be arranged in two rectangular grids. 
\end{proposition}

\begin{figure}[htb]
	\centering
	\subfigure[	Lissajous curve $\gamma_{3,2}$, $|\LubAbr_{3,2}| = 12$.]{\includegraphics[scale=0.85]{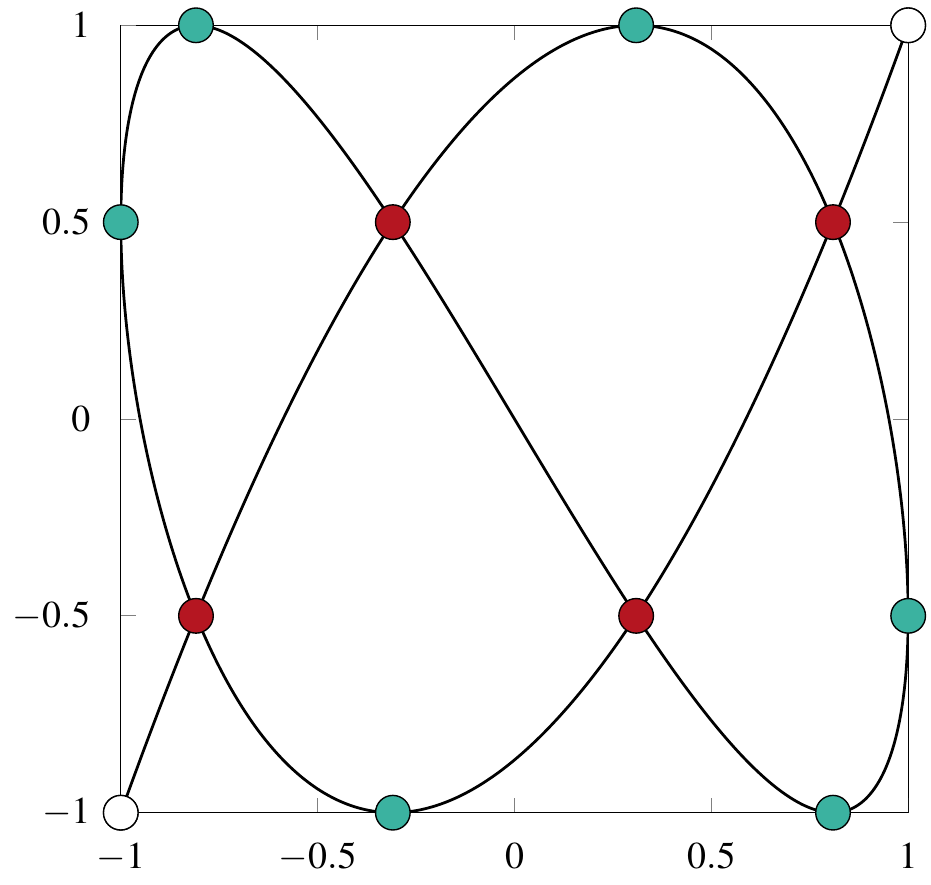}}
	\hfill	
	\subfigure[ Index set $\Gamma_{3,2}^L$, $|\Gamma_{3,2}^L| = 12|$]{\includegraphics[scale=1.05]{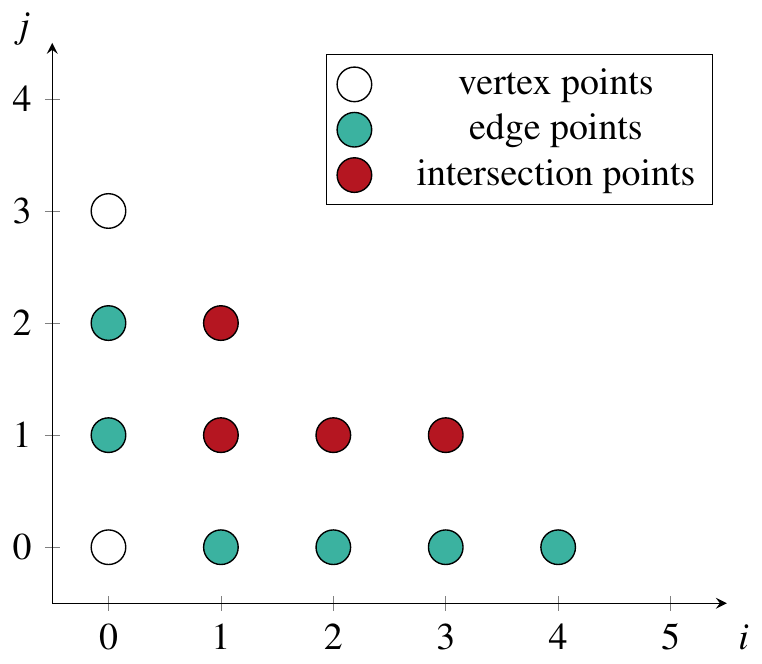}} 
  	\caption{Illustration of the degenerate Lissajous curve $\gamma_{3,2}$, its node points $\LubAbr_{3,2}$ and the corresponding index set $\Gamma_{3,2}^L$ according to the characterization \eqref{eq:LD2} in Proposition \ref{prop-2}.}
	\label{fig:lissajous1}
\end{figure}

Using \eqref{eq:generatingcurve} or \eqref{eq:chebycurve}, it is easy to see that the Lissajous curve $\Liscurve$ touches the boundary of the square $[-1,1]^2$ at exactly $2n+p$ points. We collect these boundary points in the set
\begin{equation}
\Lubout :=  \left\{  \Big( z_{i}^{n+p}, z_{j}^{n}\Big):  \begin{array}{l} i = 0, \ldots, n+p \\ j \in \{0,n\} \\ i+j = 0  \mod 2  \end{array}
\right\} \cup \left\{  \Big( z_{i}^{n+p}, z_{j}^{n}\Big):  \begin{array}{l} i \in \{0,n+p\} \\ j = 1, \ldots, n-1 \\ i+j = 0 \mod 2  \end{array}
\right\}. \label{eq:LDboundary} 
\end{equation}
We can further divide the boundary points $\Lubout$ in vertex and edge points:
\begin{align*} \Lubvert &:= \left\{\Liscurve(0), \Liscurve(\pi)\right\}, \\
\Lubedge &:= \Lubout \setminus \Lubvert.
\end{align*}   
Now, we get the following characterizations for the node set $\Lub$:

\begin{proposition} \label{prop-2}
If $n$ and $n+p$ are relatively prime, the set 
$\Lub$ contains $\frac{(n+p+1)(n+1)}{2}$ distinct points and is the union of self-intersection and boundary points
of the curve $\Liscurve$, i.e.
\begin{equation}
\Lub  = \Lubint \cup \Lubout = \left\{  \Big( z_{i}^{n+p}, z_{j}^{n}\Big): \quad \begin{array}{l} i = 0, \ldots, n+p \\ j = 0, \ldots, n \\ i+j = 0 \mod 2  \end{array}
\right\}. \label{eq:LD1} 
\end{equation}
$\Lub$ can be arranged in two rectangular grids
\begin{align} 
\Lubblack &:= \left\{  \Big( z_{i}^{n+p}, z_{j}^{n}\Big): \quad \begin{array}{l} i = 0, \ldots, n+p \\ j = 0, \ldots, n \\ \text{i,j even}  \end{array}
\right\}, \label{def:Luebeckpointsb}\\
\Lubwhite &:= \left\{  \Big( z_{i}^{n+p}, z_{j}^{n}\Big): \quad \begin{array}{l} i = 0, \ldots, n+p \\ j = 0, \ldots, n \\ \text{i,j odd}  \end{array} \right\}. \label{def:Luebeckpointsr}
\end{align}
Further, introducing the index sets
\begin{align}
\Lubindbasic &:= \left\{(i,j) \in \Nn_0^2: \; \frac{i}{n+p} + \frac{j}{n} < 1 \right\}, \label{eq:index1}\\
\Lubind &:= \Lubindbasic \cup \{(0,n)\}, \label{eq:index2}
\end{align}
the node set $\Lub$ can be characterized as
\begin{equation}
\Lub  = \left\{  \Big( z_{in+j(n+p)}^{n+p}, z_{in+j(n+p)}^{n}\Big): \; (i,j) \in \Lubind 
\right\}. \label{eq:LD2} 
\end{equation}
\end{proposition}

\begin{proof}
Clearly, $\left\{  \big( z_{in+j(n+p)}^{n+p}, z_{in+j(n+p)}^{n}\big): \; (i,j) \in \Lubind 
\right\}$ is a subset of $\Lub$. For $(i,j) \in \Lubind$, the integers $k = i n + j(n+p)$ and $k' = |in - j(n+p)|$ are nonnegative and less or equal to $n(n+p)$. 
Evaluating the Lissajous curve $\Liscurve$ at $t_k$ and $t_{k'}$, we observe that the points $\Liscurve(t_k)$ and $\Liscurve(t_{k'})$ are equal. Further, since $n$ and $n+p$ 
are relatively prime, $k$ and $k'$ coincide if and only if $i = 0$ or $j = 0$. Therefore, for all positive integers $i,j \in \Nn$ with $\frac{i}{n+p} + \frac{j}{n} < 1$,
we obtain a pair of distinct points $t_k, t_{k'} \in [0,\pi]$ such that $\Liscurve(t_k) = \Liscurve(t_{k'})$. The total number of distinct pairs is given by
$\frac{(n-1)(n+p-1)}{2}$. Therefore, by Proposition \ref{prop-1}, these pairs describe exactly all self-intersection points of the curve $\Liscurve$, i.e. 
\[ \Lubint = \left\{  \Big( z_{in+j(n+p)}^{n+p}, z_{in+j(n+p)}^{n}\Big): \; (i,j) \in \Nn^2,\; \frac{i}{n+p} + \frac{j}{n} < 1 \right\}. \]
Further, the numbers $k = in + j(n+p)$ with $(i,j) \in \Lubind$ and $i = 0$ or $j = 0$ correspond precisely with the $2n+p$ boundary points of $\Liscurve$, i.e.
\[ \Lubout = \left\{  \Big( z_{in+j(n+p)}^{n+p}, z_{in+j(n+p)}^{n}\Big): \; (i,j) \in \Lubind, \; i = 0\; \text{or}\; j = 0 \right\}. \]
Therefore,
\[ \Lubint \cup \Lubout = \left\{  \big( z_{in+j(n+p)}^{n+p}, z_{in+j(n+p)}^{n}\big): \; (i,j) \in \Lubind \right\} \subset \Lub.\]
Since $2|\Lubint| + |\Lubout| = n(n+p) + 1$, we even have equality in the last formula. This together with the definitions of $\Lubint$ and $\Lubout$ in \eqref{eq:LDinterior} and \eqref{eq:LDboundary} implies equation \eqref{eq:LD1} as well as equation
\eqref{eq:LD2}. Finally, \eqref{def:Luebeckpointsb} and \eqref{def:Luebeckpointsr} follow from \eqref{eq:LD1}. \qed
\end{proof}

Regarding the cardinality of the subgrids $\Lubblack$ and $\Lubwhite$, we can distinguish between three cases depending on whether $n$ and $p$ are even or odd integers. 
The same holds for the location of the second vertex $\Liscurve(\pi) = ((-1)^n,(-1)^{n+p})$, whereas the first vertex is always given as $\Liscurve(0) = (1,1)$. 
Using the formulas \eqref{def:Luebeckpointsb} and \eqref{def:Luebeckpointsr}, we compute the different cardinalities and 
list them in Table \ref{tab:1}. The respective cases are also illustrated in Figure \ref{fig:lissajous}.

\begin{table}[htb] 
 \caption{Number of points in the different $\LubAbr$ sets.} \label{tab:1} 
 \vspace{-0.5cm}
 \begin{center}
  \begin{tabular}[t]{c} 
  \hline \noalign{\smallskip}
  $ |\Lub| = \frac{(n+p+1)(n+1)}{2} $ \hspace{1cm}
  $ |\Lubint| = \frac{(n+p-1)(n-1)}{2} $ \hspace{1cm}
  $ |\Lubout| = 2n+p $ \\[1mm] \hline \noalign{\bigskip}
 \end{tabular}
 \begin{minipage}[t]{0.325\textwidth} \centering
  \begin{tabular}[t]{l} \hline \noalign{\smallskip}
  Case (a): $n$ even, $p$ odd \\ \noalign{\smallskip}\hline \noalign{\smallskip}
  $|\Lubblack| = \frac{n+2}2\frac{n+p+1}2 $\\ \noalign{\smallskip}
  $|\Lubwhite| = \frac{n}2\frac{n+p+1}2$\\ \noalign{\smallskip}
  $\Liscurve(\pi) = (1,-1)$ \\ \hline
  \end{tabular}
 \end{minipage}
  \begin{minipage}[t]{0.325\textwidth} \centering
  \begin{tabular}[t]{l} \hline \noalign{\smallskip}
  Case (b): $n$ odd, $p$ odd \\ \noalign{\smallskip}\hline \noalign{\smallskip}
  $|\Lubblack| = \frac{n+1}2\frac{n+p+2}2 $\\ \noalign{\smallskip}
  $|\Lubwhite| = \frac{n+1}2\frac{n+p}2$\\ \noalign{\smallskip}
  $\Liscurve(\pi) = (-1,1)$ \\ \hline
  \end{tabular}
 \end{minipage}
  \begin{minipage}[t]{0.325\textwidth} \centering
  \begin{tabular}[t]{l} \hline \noalign{\smallskip}
  Case (c): $n$ odd, $p$ even \\ \noalign{\smallskip}\hline \noalign{\smallskip}
  $|\Lubblack| = \frac{n+1}2\frac{n+p+1}2 $\\ \noalign{\smallskip}
  $|\Lubwhite| = \frac{n+1}2\frac{n+p+1}2$\\ \noalign{\smallskip}
  $\Liscurve(\pi) = (-1,-1)$ \\ \hline
  \end{tabular}
 \end{minipage}
   \end{center}
\end{table}

\begin{figure}[htb]
	\centering
	\subfigure[	$\gamma_{2,3}$, $|\LubAbr_{2,3}| = 9$.]{\includegraphics[scale=0.55]{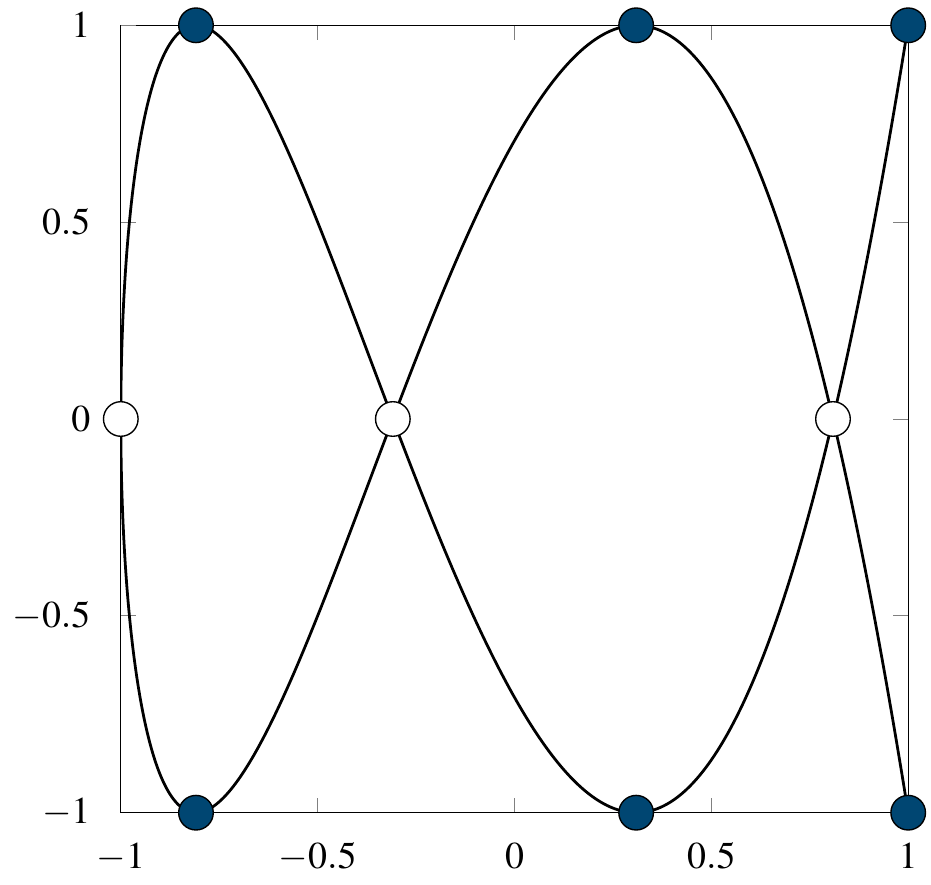}}
	\hfill	
	\subfigure[	$\gamma_{3,1}$, $|\LubAbr_{3,1}| = 10$.]{\includegraphics[scale=0.55]{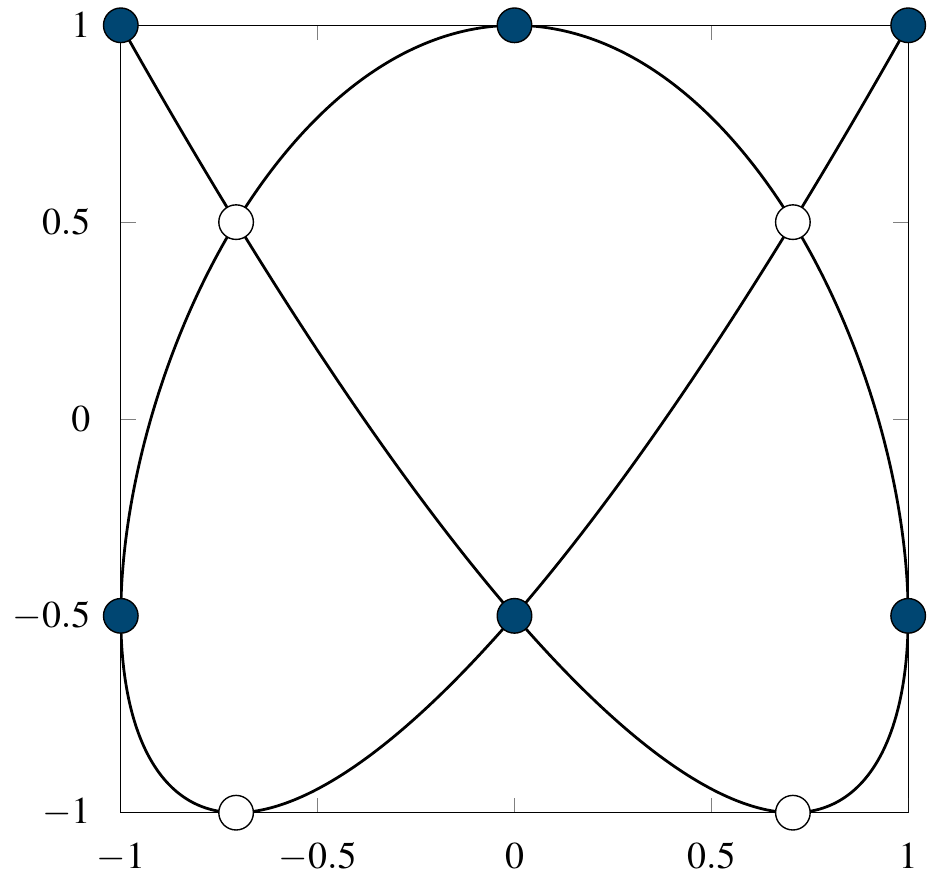}} \hfill
	\subfigure[	$\gamma_{3,2}$, $|\LubAbr_{3,2}| = 12$.]{\includegraphics[scale=0.55]{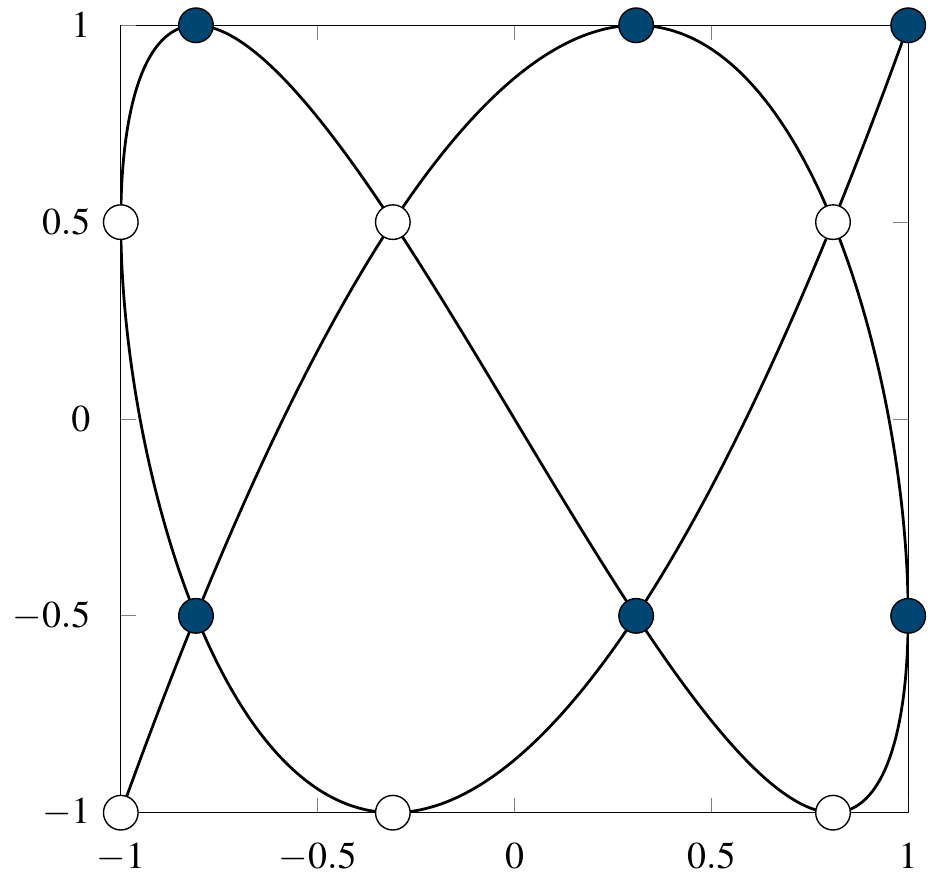}}
  	\caption{Illustration of degenerate Lissajous curves $\Liscurve$ and its node points $\Lub$ according to the cases in Table \ref{tab:1}. The points in the subgrids $\Lubblack$ and $\Lubwhite$ are colored in blue and white, respectively. 
  	}
	\label{fig:lissajous}
\end{figure}

To simplify the notation of the integers $k$ in \eqref{def:Luebeckpoints} that describe the same
point $\Aa \in \Lub$, we introduce on $\Zz$ the equivalence relation $\equi$ by
\begin{equation} \label{eq:equivalence}
k \equi k' \quad \Leftrightarrow \quad \Liscurve (t_k) = \Liscurve (t_{k'}).
\end{equation}
In this way, we obtain for each $\Aa \in \Lub$ a unique equivalence class $[\Aa]$ and we say that $k \in [\Aa]$ if $\Liscurve(t_k) = \Aa$. 
By the argumentation in the proof of Proposition \ref{prop-2}, there is exactly one ${0 \leq k \leq n(n+p)}$ in the equivalence class $[\Aa]$ if $\Aa \in \Lubout$
and exactly two if $\Aa \in \Lubint$ is a self-intersection point. 

\begin{remark}
In this article, the parameter choice $p = 1$ corresponds to the well-known Padua points studied extensively in 
\cite{BosDeMarchiVianelloXu2006,BosDeMarchiVianello2006,BosDeMarchiVianelloXu2007,CaliariDeMarchiSommarivaVianello2011,CaliariDeMarchiVianello2005,CaliarideMarchiVianello2008-2,CaliariDeMarchiVianello2008}.
For the Padua points, one can differ between four families of interpolation points by considering ninety degree rotations of the set $\LubAbr_{n,1}$. Similarly, one gets four families of interpolation nodes for general $p$. In the cases (a) and (b) considered in Table \ref{tab:1} and Figure \ref{fig:lissajous} this is also done by rotating $\Lub$ by ninety degrees. In the case (c) one has to combine ninety degree rotations and reflections with respect to one of the coordinate axes to obtain the four families. The respective generating Lissajous figure is a rotated and reflected version of $\Liscurve$. The interpolation theory developed in this paper can be applied to all four families of nodes. For simplicity, we will only consider the family $\Lub$ generated by $\Liscurve$.
\end{remark}

\begin{remark}
The set $\Lub$ is a two-dimensional Chebyshev lattice of rank $1$. According to the notation given in \cite{CoolsPoppe2011,PoppeCools2013}, the respective parameters of the Chebyshev lattice are $d = n(n+p)$ and $\mathbf{z} = [n,n+p]$. The following sections will show that lattices of this form allow unique polynomial interpolation.  
\end{remark}

\section{Lissajous node points and quadrature}

The point set $\Lub$ can be used for interpolation purposes as well as for quadrature rules on $[-1,1]^2$. 
In this section, we study first a quadrature formula based on point evaluations at the set $\Lub$. Of essential importance in our considerations is the operator
\[ \Emb: C([-1,1]^2) \to C([0,\pi]), \quad \Emb \! f(t) = f(\Liscurve(t)), \quad t \in [0,\pi],\]
that restricts a continuous function on $[-1,1]^2$ to the Lissajous trajectory $\Liscurve$. It is clear that $\Emb$ maps 
bivariate algebraic polynomials to even trigonometric polynomials on the interval $[0,\pi]$. 

To specify the spaces of bivariate polynomials we introduce
\begin{equation} \label{eq:spacebivariatepolynomials} \Pin = \spann\{T_{i}(x) T_j(y):\; i+j \leq n\},\end{equation}
where $T_i(x) =  \cos (i \arccos x)$, as before, denote the Chebyshev polynomials of the first kind. It is well-known (cf. \cite{Xu1996}) that $\{ T_i(x) T_j(y):\; i+j \leq n\}$ forms an orthogonal basis for the space $\Pin$ with respect to the inner product 
\begin{equation} \label{eq:scalarproduct}
 \langle f,g \rangle := \frac{1}{\pi^2} \int_{-1}^1 \int_{-1}^1 f(x,y) \overline{g(x,y)} \frac{1}{\sqrt{1-x^2}} \frac{1}{\sqrt{1-y^2}} \Dx{x} \Dx{y}.
\end{equation}
With the normalization
\begin{equation*} \label{eq:normalizedpolynomials} \hat{T}_i(x) = \left\{ \begin{array}{ll} 
                          1, & \quad \text{if $i = 0$},  \\
                          \sqrt{2} T_i(x), & \quad \text{if $i \neq 0$},
                          \end{array} \right.
\end{equation*}
we obtain the orthonormal basis $\{\hat{T}_i(x) \hat{T}_j(y): i+j \leq n \}$ of $\Pin$.

The first auxiliary result shows that for a large class of bivariate polynomials $P$ the restriction $\Emb P $ can be used to convert a double integral into a one dimensional integral of a trigonometric polynomial.
\begin{lemma} \label{lem:int2Dto1D}
For all bivariate polynomials $P$ with $\langle P, T_{k(n+p)}(x)T_{k n}(y) \rangle = 0$, $k \in \Nn$, the following formula holds:
\begin{equation} \label{eq:int2Dto1D}
\frac{1}{\pi^2} \int_{-1}^1 \int_{-1}^1 P(x,y) \frac{1}{\sqrt{1-x^2}} \frac{1}{\sqrt{1-y^2}} \Dx{x} \Dx{y} = \frac1{\pi} \int_0^{\pi} \Emb P(t) \Dx{t}. \end{equation}
\end{lemma}

\begin{proof} Lemma \ref{lem:int2Dto1D} is a slight generalization of \cite[Lemma 1]{BosDeMarchiVianelloXu2006}.
With minor changes, the proof follows the same lines of argumentation as the proof of \cite[Lemma 1]{BosDeMarchiVianelloXu2006} and \cite[Lemma 1]{ErbKaethnerAhlborgBuzug2014}. \qed
\end{proof}

Next, we consider the spaces of bivariate polynomials corresponding to the index sets $\Lubindbasic$ and $\Lubind$ introduced in
\eqref{eq:index1} and \eqref{eq:index2}:

\begin{equation} \label{def:polynomialinterpolationspace}
\begin{array}{ll} \ds \Pibasic &:= \spann\{T_{i}(x) T_j(y):\; (i,j) \in \Lubindbasic \}, \\[2mm]
\ds \PiL     &:= \spann\{T_{i}(x) T_j(y):\; (i,j) \in \Lubind \}. 
\end{array}
\end{equation}

By the characterization of the set $\Lub$ given in \eqref{eq:LD2} it follows immediately that 
\begin{align*}
 \dim \PiL &= |\Lubind| = |\Lub| = \frac{(n+p+1)(n+1)}{2},\\
 \dim \Pibasic &= |\Lubindbasic| = |\Lubind|-1 = \frac{(n+p+1)(n+1)}{2}-1.
\end{align*}
To characterize the range of the operator $\Emb$, we need particular spaces of trigonometric polynomials on $[0,\pi]$:
\[\Pi_{N}^{\mathrm{e}} := \left\{ q(t) = \sum_{m=0}^{N} a_m \cos(m t) :\quad a_m \in \Rr, \; t \in [0,\pi] \right\}.\] 

\begin{lemma} \label{lem-isomorphism1} The operator $\Emb$ defines an isometry 
from the polynomial space $\Pibasic$ ($\PiL$) equipped with the inner product $\langle \cdot, \cdot \rangle$ given in \eqref{eq:scalarproduct} into the trigonometric space
$\Tripibasicb$ ($\Tripibasica$, respectively) equipped with the inner product $ \ds \langle q_1, q_2 \rangle = \frac1{\pi} \int_0^{\pi} \! q_1(t) q_2(t) \Dx{t}$.
\end{lemma}

\begin{proof}
An orthonormal basis of $\Pibasic$ is given by $\left\{ \hat{T}_i(x) \hat{T}_j(y): (i,j) \in \Lubindbasic \right\}$. The image 
\begin{equation*} 
\Lagbastrig(t) := \LagbastrigE(t), \quad (i,j) \in \Lubindbasic,
\end{equation*}
of this orthonormal basis is explicitly given by
\begin{equation} \label{eq-orthogonaltrigonometricbasis} 
                                          \Lagbastrig(t) =  \left\{\begin{array}{ll} 
                                          \, 1, & \text{if}\; (i,j) = (0,0), \\[1mm]
                                          \sqrt{2}  \cos \left(i n t \right), & \text{if}\; j=0, i < n+p, \\[1mm]
                                          \sqrt{2}  \cos \left(j (n+p) t \right), & \text{if}\; i=0, j < n, \\[1mm]
                                          \,2 \cos \left(i n t \right) \cos \left(j (n+p) t \right) ,\quad & \text{for all other $(i,j) \in \Lubindbasic$}. 
                                          \end{array} \right.
\end{equation}
From the definition \eqref{eq:index1} of the index set $\Lubindbasic$, we obtain $in + j(n+p) < n(n+p)$ for $(i,j) \in \Lubindbasic$. Therefore, all $\Lagbastrig$ 
are trigonometric polynomials in the space $\Tripibasicb$ and we can conclude that $\Emb$ maps $\Pibasic$ into the space $\Tripibasicb$. 

Since $\Lubind = \Lubindbasic \cup \{(0,n)\}$ and $e_{0,n}(t) = \Emb (\hat{T}_n(y)) (t) = \sqrt{2}  \cos \left( n (n+p) t \right)$ is a trigonometric polynomial
of degree $n(n+p)$, we can also conclude that
$\Emb$ maps $\PiL$ into the space $\Tripibasica$

For $P_1,P_2 \in \PiL$, the product $P_1 P_2$ is a polynomial in the space $\Pi_{2n+2p-2}^2$ and satisfies $\langle P_1 P_2, T_{n+p}(x)T_{n}(y) \rangle = 0$. Therefore, 
by Lemma \ref{lem:int2Dto1D}, the set $\left\{ \Lagbastrig: \; (i,j) \in \Lubind \right\}$ is an orthonormal system in $\Tripibasica$, and thus, $\Emb$ an isometry
from $\PiL$ into the space $\Tripibasica$. By the same argumentation, the set $\left\{ \Lagbastrig: \; (i,j) \in \Lubindbasic \right\}$ 
is an orthonormal system in $\Tripibasicb$. This proves the isometry also for the space $\Pibasic$. \qed
\end{proof}

For points $\Aa \in \Lub$, we introduce the quadrature weights
\[w_{\Aa} := \left\{ \begin{array}{ll} 
\ds \frac{1}{2n(n+p)}, \quad & \text{if $\Aa \in \Lubvert$}, \\ \ds \frac{1}{n(n+p)}, \quad & 
\text{if $\Aa \in \Lubedge$}, \\
                             \ds \frac{2}{n(n+p)}, \quad & \text{if $\Aa \in \Lub^\mathrm{int}$},
                 
                \end{array}
\right.\]
and obtain the following quadrature rule for the node set $\Lub$: 

\begin{theorem} \label{thm:quadratureruleluebeck}
For all polynomials $P \in \Pi_{2n,2p}^2$ the quadrature formula
\begin{equation} \label{eq:quadratureruleluebeck}
\frac{1}{\pi^2} \int_{-1}^1 \int_{-1}^1 P(x,y) \frac{1}{\sqrt{1-x^2}} \frac{1}{\sqrt{1-y^2}} \Dx{x} \Dx{y} = \sum_{\Aa \in \Lub} w_\Aa P(\Aa) 
\end{equation}
is exact.
\end{theorem}

\begin{proof}
For all even trigonometric polynomials $q \in \Tripibasicc$ the following composite trapezoidal quadrature rule is exact (see \cite[Chapter X]{Zygmund}):
\[ \frac{1}{\pi} \int_0^{\pi} q(t) dt = \frac1{n(n+p)} \left( \frac12 q(0) + \sum_{k=1}^{n(n+p)-1} \hspace{-0.3cm} q\left( t_k \right) + \frac12 q(\pi)\right).\]
Since by the definitions \eqref{eq:spacebivariatepolynomials} and \eqref{def:polynomialinterpolationspace} of the polynomial spaces $\Pi^2_{2n+2p-1}$ and $\Pi_{2n,2p}^2$ we have $\Pi_{2n,2p}^2 \subset \Pi^2_{2n+2p-1}$ and $ \Pi^2_{2n,2p} \perp T_{n+p}(x)T_{n}(y)$, Lemma \ref{lem:int2Dto1D} yields the identity
\[
\frac{1}{\pi^2} \int_{-1}^1 \int_{-1}^1 P(x,y) \frac{1}{\sqrt{1-x^2}} \frac{1}{\sqrt{1-y^2}} \Dx{x} \Dx{y} = \frac1{\pi} \int_0^{\pi} \Emb P (t) \Dx{t}.
\]
For a polynomial $P \in \Pi^2_{2n,2p}$, the image $\Emb P$ is by a similar argumentation as in the proof of Lemma \ref{lem-isomorphism1} an element of $\Tripibasicc$. 
Therefore, combining the two identities above and using the definition \eqref{def:Luebeckpoints} of the points $\Lub$ as well as its characterizations in 
Proposition \ref{prop-2}, we get the quadrature formula
\begin{align*}
\frac{1}{\pi^2} \int_{-1}^1 \int_{-1}^1 P(x,y) \frac{1}{\sqrt{1-x^2}} \frac{1}{\sqrt{1-y^2}} \Dx{x} \Dx{y} &= \sum_{\Aa \in \Lub} w_\Aa P(\Aa).
\end{align*}
 \qed
\end{proof}

\begin{remark}
For the case $p = 1$ of the Padua points, Theorem \ref{thm:quadratureruleluebeck} is proven in a similar way in \cite{BosDeMarchiVianelloXu2006}. Quadrature formulas similar to \eqref{eq:quadratureruleluebeck} exist also for other related point sets as the Xu points (see \cite{Harris2010,MorrowPatterson1978,Xu1996}). 
The operator $\Emb$ was already introduced in \cite{ErbKaethnerAhlborgBuzug2014} to describe the relation between algebraic polynomials on $[-1,1]^2$ and trigonometric polynomials on the Lissajous trajectory. However, due to the geometric differences between degenerate and non-degenerate Lissajous curves, the range of the operator $\Emb$ consists of different trigonometric polynomials in the two cases. 
\end{remark}

\section{Lissajous node points and interpolation}

In this section, the object of study is the following interpolation problem: for given 
data values $f_\Aa \in \Rr$ at the node points $\Aa \in \Lub$, find a unique interpolating polynomial $\Lagpol$ in $[-1,1]^2$ such that 
\begin{equation}\label{eq:interpolationproblem}
\Lagpol (\Aa) = f_\Aa \quad \text{holds for all}\; \Aa \in \Lub.
\end{equation}

In the bivariate setting it is a priori not clear which polynomial space has to be chosen for the interpolation problem \eqref{eq:interpolationproblem}. Since $\dim \PiL = |\Lubind| = |\Lub|$, our primary choice is the polynomial space $\PiL$ defined in \eqref{def:polynomialinterpolationspace}. 
To prove the uniqueness of \eqref{eq:interpolationproblem} in $\PiL$, we study first an isomorphism between $\PiL$ and the subspace
\begin{equation} \label{def:spacetrigonometricpolynomials}
 \Tripi := \left\{  q \in \Tripibasica:  
\; \text{$q(t_k) = q(t_{k'})$ for all $k,k'$ with $k \!\equi\! k'$} \right\} 
\end{equation}
of the even trigonometric polynomials $\Tripibasica$.

\begin{theorem} \label{Thm-isomorphism} The operator $\Emb$ is an isometric isomorphism from $\PiL$ onto the subspace
$\Tripi$, equipped with the inner product given in Lemma \ref{lem-isomorphism1}.
\end{theorem}

\begin{proof}
By Lemma \ref{lem-isomorphism1}, we know that $\Emb$ is an isometry from $\PiL$ into the space $\Tripibasica$. 
If $\Liscurve(t_k) \in \Lubint$ is a self-intersection point of the Lissajous curve $\Liscurve$, then $\Liscurve(t_{k'}) = \Liscurve(t_{k})$ holds for a $ 0 \leq k' \leq n(n+p)$, $k' \neq k$ and the values $\Emb\! P (t_k)$ and $\Emb\! P (t_{k'})$ coincide. 
This property is exactly encoded in the definition \eqref{def:spacetrigonometricpolynomials} of the space $\Tripi$. This implies that the operator $\Emb$ maps $\PiL$ into the subspace $\Tripi$. 

Now, it suffices to show that the dimensions of $\PiL$ and $\Tripi$ coincide. Then, we immediately obtain the surjectivity and, hence,
the bijectivity of $\Emb$. To this end, we consider in $\Tripibasica$ the Dirichlet kernel (see \cite[X Section $3$] {Zygmund})
\[D_{n(n+p)}(t) := 1 + \ds 2 \!\!\! \sum_{k=1}^{n(n+p)-1} \!\!\!\! \cos(kt) \; + \cos (n(n+p)t) = \frac{\sin(n(n+p) t ) \cos \frac{t}2 }{\sin \frac{t}2} .\]
The translates $\frac{1}{n(n+p)} D_{n(n+p)}(t-t_k)$, $k = 1, \ldots, 2n(n+p)$, form a 
basis for the space of trigonometric polynomials of the form $q(t) = \sum_{m=0}^{n(n+p)} a_m \cos(m t) + \sum_{m=0}^{n(n+p)-1} b_m \sin(mt)$. Therefore, for the space $\Tripibasica$ of even trigonometric polynomials we get
\[ D_{n(n+p)}^k(t) := \left\{ 
\begin{array}{ll} \frac{D_{n(n+p)}\left(t \right)}{n(n+p)}, & \text{if $k=0$,} \\
                  \frac{D_{n(n+p)}\left(t + t_k\right) + D_{n(n+p)}\left(t - t_k \right)}{n(n+p)}, \quad &  \text{if $k = 1, \ldots, n(n+p) - 1$,} \\ 
                  \frac{D_{n(n+p)}\left(t - \pi \right)}{n(n+p)} , & \text{if $k=n(n+p)$,}
\end{array} \right.\]
as a fundamental basis of Lagrange polynomials with respect to the equidistant points $t_k$, $k = 0, \ldots, n(n+p)$, i.e.
\[D_{n(n+p)}^k \left( t_{k'} \right) = \delta_{k,k'}, \quad 0 \leq k,k' \leq n(n+p).\]
Since not all $D_{n(n+p)}^k$ are contained in the subspace $\Tripi$, we introduce for $\Aa \in \Lub$ the linear combinations
\begin{equation} \label{eq:Lagrangebasispolynomialstrigonometric}
 l_\Aa(t) := \sum_{\substack{0 \leq k \leq n(n+p): \\ k \in [\Aa]}} D_{n(n+p)}^k(t),
\end{equation}
where $[\Aa]$ denotes the equivalence class introduced in \eqref{eq:equivalence}. The trigonometric polynomials $l_\Aa$ are elements of the space $\Tripi$ and 
\begin{equation} \label{eq:Lagrangebasispolynomialstrigonometricproperties}
  l_{\Aa}(t_k) = \left \{ \begin{array}{ll} 1 \quad & \text{if $k \in [\Aa]$,} \\ 0 &  \text{if $k \notin [\Aa]$.} \end{array} \right.
\end{equation}
Since $D_{n(n+p)}^k$, $0 \leq k \leq n(n+p)$, form a basis in the space $\Tripibasica$, the system $\{l_\Aa: \; \Aa \in \Lub \}$ is a basis in the linear subspace $\Tripi$. This implies the desired equality $\dim (\Tripi) = |\Lub| = \dim (\PiL)$.\qed
\end{proof}

Finally, we prove the uniqueness of the interpolation problem \eqref{eq:interpolationproblem} in $\PiL$ and give explicit formulas 
for the corresponding fundamental Lagrange polynomials. As a final auxiliary tool, we need the reproducing kernel $K_{n,p}^L: \Rr^2 \times \Rr^2 \to \Rr$ 
of the polynomial space $\PiL$. It is defined as
\begin{equation} \label{eq:reproducingkernel}
K_{n,p}^L(x_\Aa,y_\Aa; x_\Bb, y_\Bb) := \sum_{(i,j) \in \Gamma_{n,p}^L} \hat{T}_i(x_\Aa) \hat{T}_i(x_\Bb) \hat{T}_j(y_\Aa) \hat{T}_j(y_\Bb).
\end{equation}
With $K_{n,p}$ we denote the respective reproducing kernel for the subspace $\Pi_{n,p}^2$.

\begin{theorem} \label{thm:interpolation problem}
For $\Aa = (x_\Aa, y_\Aa) \in \Lub$, the polynomials  $L_{\Aa} := \Emb^{-1} l_{\Aa}$ are given as
\begin{equation}
 \label{eq:fundamentalpolynomialsofLagrangeinterpolation}
 \begin{array}{ll}
 L_{\Aa}(x,y) & \!\! = \, \displaystyle w_\Aa \left( K_{n,p}^L(x,y; x_\Aa,y_\Aa) - {\textstyle \frac12} \hat{T}_{n}(y) \hat{T}_{n}(y_\Aa) \right) \\[2mm]
 & \!\! = \, \displaystyle w_\Aa \left( K_{n,p}(x,y; x_\Aa,y_\Aa) + {\textstyle \frac12 } \hat{T}_{n}(y) \hat{T}_{n}(y_\Aa) \right) \end{array}
\end{equation}
and form the fundamental Lagrange polynomials in the space $\PiL$ with respect to the point set $\Lub$, i.e. 
\[ L_{\Aa}(\Bb) = \delta_{\Aa,\Bb}, \quad \Aa, \Bb \in \Lub.\]
The interpolation problem \eqref{eq:interpolationproblem} has a unique solution in $\PiL$ given by
\begin{equation*} \label{eq:interpolationpolynomial}
\Lagpol (x,y) = \sum_{\Aa \in \Lub} f_\Aa L_\Aa(x,y). 
\end{equation*}
\end{theorem}

\begin{proof}
Theorem \ref{Thm-isomorphism} and property \eqref{eq:Lagrangebasispolynomialstrigonometricproperties} of the basis functions $l_{\Aa} \in \Tripi$ introduced in \eqref{eq:Lagrangebasispolynomialstrigonometric} imply that
the polynomials $L_{\Aa} = \Emb^{-1} l_{\Aa}$ satisfy $L_{\Aa}(\Bb) = \delta_{\Aa,\Bb}$ for $\Bb \in \Lub$. 
Moreover, since the system ${\{l_\Aa: \; \Aa \in \Lub \}}$ is a basis in $\Tripi$, Theorem \ref{Thm-isomorphism} implies that also the system 
$\{L_\Aa: \; \Aa \in \Lub \}$ forms a basis of Lagrange polynomials in the space $\PiL$. 

Finally, we compute the decomposition of the trigonometric polynomials $l_\Aa$ in the orthonormal basis $e_{i,j}$.  Then, using the inverse $\Emb^{-1}$, we get the representation \eqref{eq:fundamentalpolynomialsofLagrangeinterpolation} for the Lagrange polynomials $L_\Aa$. 

For the calculations, we use the characterization \eqref{eq:LD2} for points $\Aa \in \Lub$, i.e. $\Aa = (x_\Aa, y_\Aa) = (z_{k}^{n+p}, z_{k}^{n})$  with $k = rn + s(n+p)$ and $(r,s) \in \Lubind$. 

We take first a look at the self-intersection points $\Aa \in \Lubint$.
From the proof of Proposition \ref{prop-2}, we know that the second integer $0 \leq k' \leq n(n+p)$ representing the self-intersection point $\Aa$ is given by $k' = |rn - s(n+p)|$. 
In this way, for the functions $l_\Aa$, $\Aa \in \Lubint$ we get
\begin{align*}
l_\Aa(t) &= \frac{1}{n(n+p)} \left( D_{n(n+p)}^k(t) + D_{n(n+p)}^{k'}(t) \right) \\ &= \frac{2}{n(n+p)} 
\left(  1 +  \sum_{m=1}^{n(n+p)}   (2-\delta_{m,n(n+p)}) 
                        \left( \cos \frac{k m \pi}{n (n+p)}+ \cos \frac{k' m \pi}{n (n+p)}\right) \cos m t  \right) \\
                        &= \frac{2}{n(n+p)} 
\left(  1 +  \sum_{m=1}^{n(n+p)}   (2-\delta_{m,n(n+p)}) 
                        \cos \frac{r m \pi}{n+p} \cos \frac{s m \pi}{n} \cos m t  \right),                                                              
\end{align*}
where $\delta_{i,j}$ denotes the Kronecker delta, i.e. $\delta_{i,j} = 1$ if $i = j$ and $\delta_{i,j} = 0$ otherwise. 
Analogous calculations for the vertex points $\Aa \in \Lubvert$ and the edge points $\Aa \in \Lubedge$ yield the same formula for $l_\Aa$ differing only in the weight $w_\Aa$. So, for general $\Aa \in \Lub$, we get
\[l_\Aa(t) = w_\Aa 
\left(  1 + \sum_{m=1}^{n(n+p)}  (2-\delta_{m,n(n+p)}) 
                        \cos \frac{r m \pi}{n+p} \cos \frac{s m \pi}{n} \cos m t  \right).  \]
Now, using formula \eqref{eq-orthogonaltrigonometricbasis} for the 
basis functions $\Lagbastrig$, we get after a small calculation the following formula for the coefficients $\left\langle l_\Aa, \Lagbastrig \right\rangle = \frac{1}{\pi} \int_0^{\pi} l_\Aa (t) \Lagbastrig(t) \Dx{t}$:
\begin{align*}
\left\langle  l_\Aa, \Lagbastrig \right\rangle &= \left\{
\begin{array}{ll}
w_\Aa, & \text{if}\; (i,j) = (0,0), \\
w_\Aa \frac1{\sqrt{2}} \cos k \pi, & \text{if}\;i = 0, j = n, \\
w_\Aa \sqrt{2} \cos \frac{ k j \pi }{n} , & \text{if}\; i = 0, j < n, \\
w_\Aa \sqrt{2} \cos  \frac{k i \pi }{n+p}, & \text{if}\; i \neq 0, j = 0,\\
w_\Aa 2 \cos \frac{ k i \pi }{n+p} \cos \frac{ k j \pi }{n}, \quad &  \text{for all other $(i,j) \in \Lubind$,}
\end{array} \right. \\[2mm]
 &= w_\Aa \left\{ 
\begin{array}{ll}
\frac1{2}\, \hat{T}_{n}(y_\Aa), & \text{if}\; i = 0, j = n, \\[2mm]
\hat{T}_{i}(x_\Aa) \hat{T}_j(y_\Aa),\quad\quad\; & \text{if}\; (i,j) \in \Lubindbasic.
\end{array} \right.
\end{align*}

Therefore, the decomposition of $l_\Aa$ in the basis $e_{i,j}$ can be written as
\[l_\Aa(t) = \frac{w_\Aa}{2} \hat{T}_{n}(y_\Aa) e_{0,n}(t)\;  + \!
\sum_{\;\;(i,j) \in \Lubindbasic} \!\! w_\Aa \hat{T}_{i}(x_\Aa) \hat{T}_j(y_\Aa)  \Lagbastrig(t). \]
Now, the inverse mapping $\Emb^{-1}$ and the definition \eqref{eq:reproducingkernel} of $K_{n,p}^L$ yield
\[L_\Aa(x,y) = \Emb^{-1} l_\Aa (x,y) = w_\Aa \left( K_{n,p}^L(x,y; x_\Aa, y_\Aa) - \frac12 \hat{T}_{n}(y) \hat{T}_{n}(y_\Aa) \right). \]
\qed
\end{proof}

\begin{figure}[htb]
	\centering
	\subfigure[	Lagrange polynomial $L_\Aa$, $\Aa = (z_4^{11}, z_6^9)$.]{\includegraphics[scale=0.56]{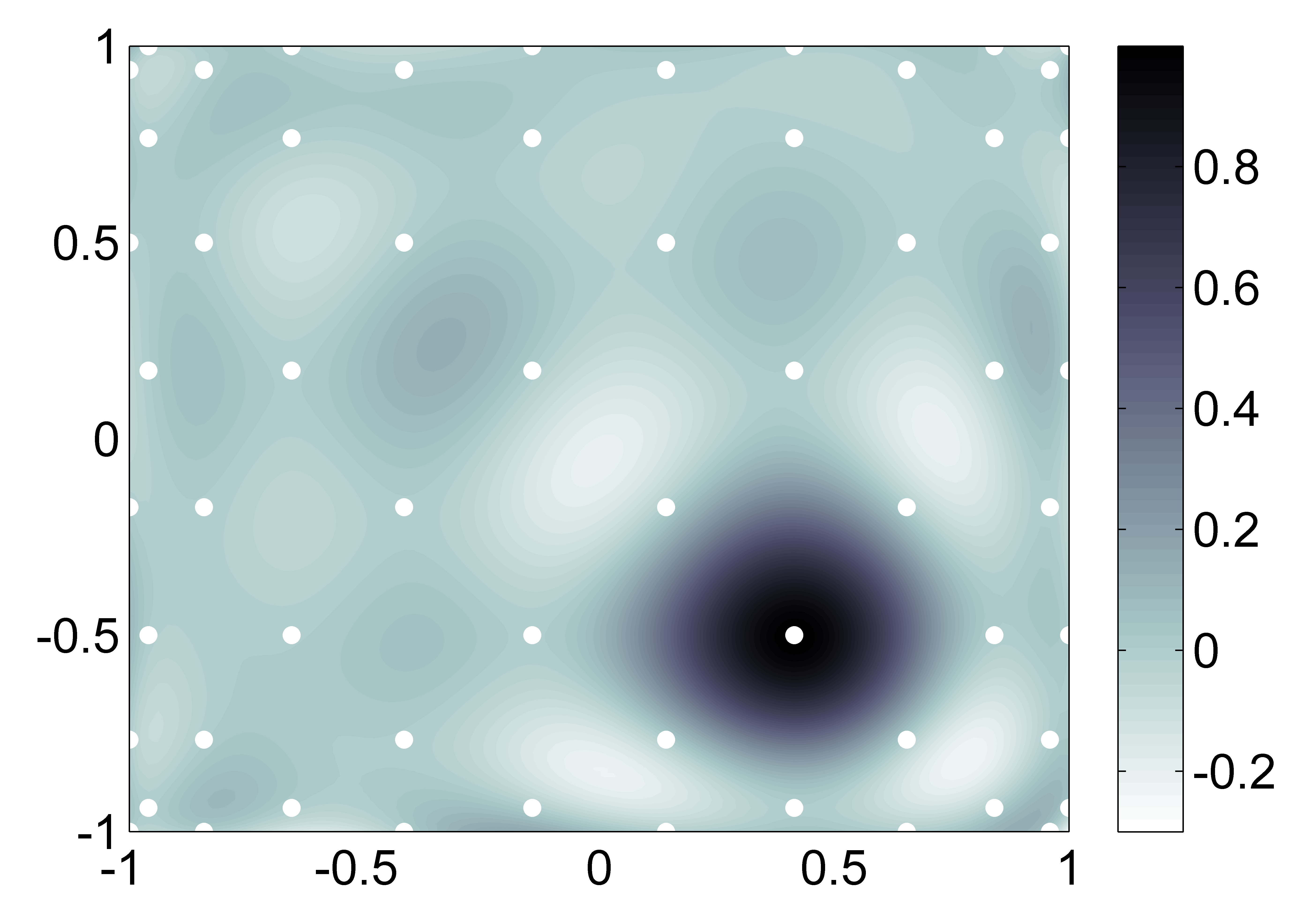}}
	\hfill	
	\subfigure[ Lagrange polynomial $L_\Aa$, $\Aa = (z_{10}^{11}, z_4^9)$.]{\includegraphics[scale=0.56]{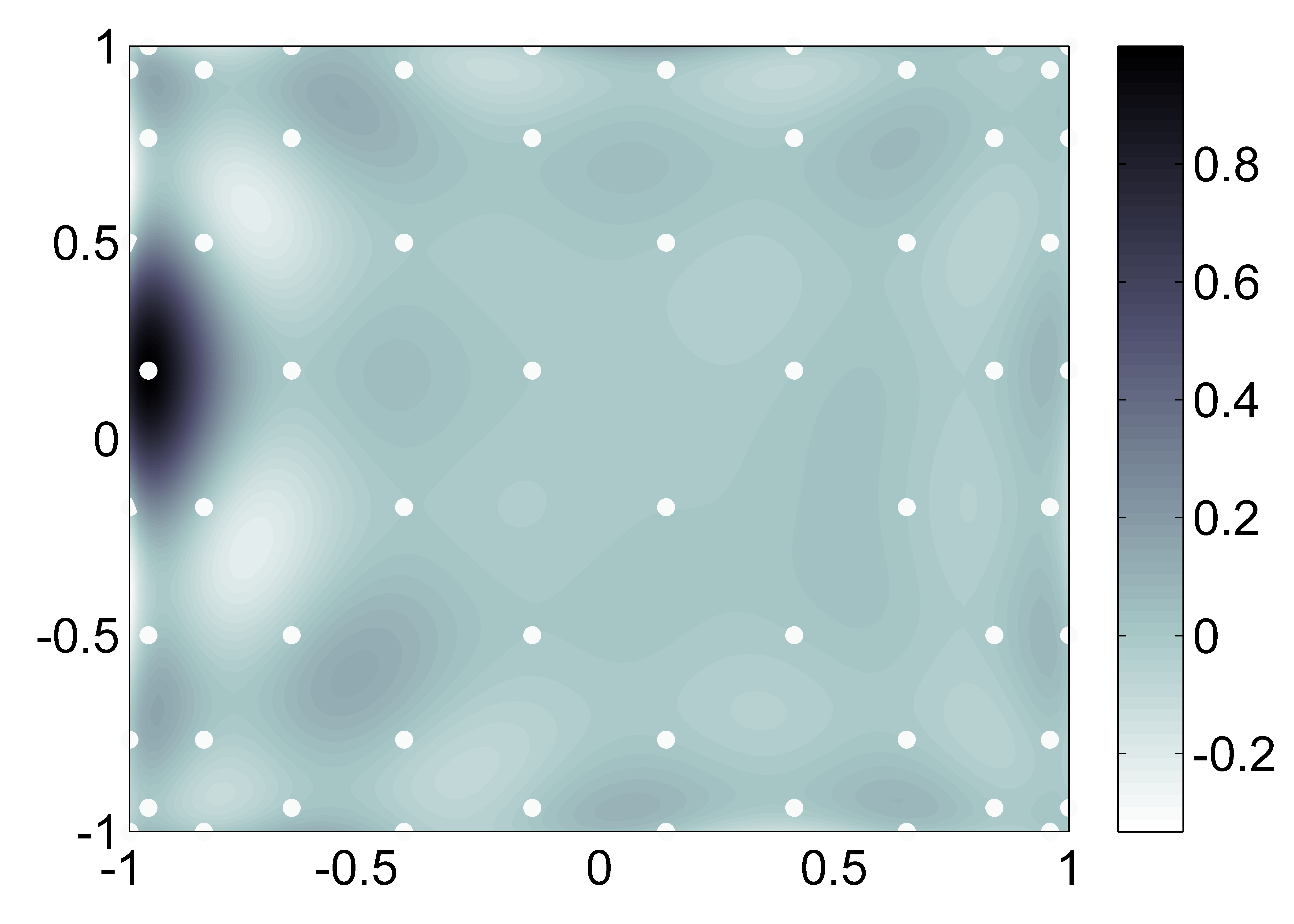}} 
  	\caption{Space localization of the Lagrange polynomials $L_{\Aa} \in \Pi_{9,2}^{2,L}$ for different $\Aa \in \LubAbr_{9,2}$.}
	\label{fig:lagrangebasis}
\end{figure}

Formula \eqref{eq:fundamentalpolynomialsofLagrangeinterpolation} in Theorem \ref{thm:interpolation problem} 
allows to compute the fundamental Lagrange polynomials $L_\Aa$ in an explicit way. 
Two examples of $L_\Aa$ are illustrated in Figure \ref{fig:lagrangebasis}. 
Moreover, using formula \eqref{eq:fundamentalpolynomialsofLagrangeinterpolation} we can rewrite the interpolating polynomial $\Lagpol(x,y)$ in terms of the orthonormal Chebyshev basis $\{ \hat{T}_i(x) \hat{T}_j(y):\; (i,j) \in \Gamma_{n,p}^L\}$. 
In this way, we obtain the representation
\[ \Lagpol(x,y) = \sum_{(i,j) \in \Gamma_{n,p}^L} c_{i,j} \hat{T}_i(x) \hat{T}_j(y)\]
with the Fourier-Lagrange coefficients $c_{i,j} = \langle \Lagpol, \hat{T}_i(x) \hat{T}_j(y) \rangle$ given by
\begin{equation} \label{eq:coeffinterpolation}
 c_{i,j} = \left\{ \begin{array}{ll}
  \ds \!\!\! \sum_{\;\;\;\Aa \in \Lub} \!\!\! w_\Aa f_\Aa \, \hat{T}_i(x_\Aa) \hat{T}_j(y_\Aa),\; & \text{if}\; (i,j) \in \Lubindbasic, \\[2mm]
                                                                          \ds \frac12 \!\!\! \sum_{\;\;\;\Aa \in \Lub} \!\!\! w_\Aa f_\Aa \, \hat{T}_{n}(y_\Aa), & \text{if} \; (i,j) = (0,n). \\
                                                                          \end{array}
\right.
\end{equation}

The computation of the coefficients $c_{i,j}$ in \eqref{eq:coeffinterpolation} can be formulated more compactly using the matrix notation
\begin{equation} \label{eq:compcoeff} \Ccc_{n,p} = \left( \Ttt_x(\Lub) \,\Ddd_f(\Lub)\, \Ttt_y(\Lub)^T \right) \odot \Mmm_{n,p},\end{equation}
where $\odot$ denotes pointwise multiplication of the single matrix entries. Here, the matrix $\Ccc_{n,p} = (c_{i,j}) \in \Rr^{(n+p) \times (n+1)}$ contains the coefficients
\[ c_{i,j} = \left\{ \begin{array}{ll} \langle \Lagpol, \hat{T}_i(x) \hat{T}_j(y) \rangle,\quad & \text{if $(i,j) \in \Gamma_{n,p}^L$,}\\ 0, & \text{otherwise}.
                       \end{array} \right.\]
The given data $f_\Aa$ and the weights $w_\Aa$ are arranged in the diagonal matrix
\[ \Ddd_f(\Lub) = \diag \left( w_\Aa f_\Aa,\; \Aa \in \Lub \right) \in \Rr^{|\Lub|\times|\Lub|}.\]
Further, all evaluations $\hat{T}_i(x_\Aa)$ and $\hat{T}_j(y_\Aa)$ are collected in the matrices
\begin{align*} \Ttt_x(\Lub) &= \underbrace{\begin{pmatrix}
                   &\hat{T}_0(x_\Aa) & \\ \cdots & \vdots & \cdots \\ &\hat{T}_{n+p-1}(x_\Aa)& 
                  \end{pmatrix}}_{\Aa \in \Lub} ,\qquad \Ttt_y(\Lub) = \underbrace{\begin{pmatrix}
                   &\hat{T}_0(y_\Aa) & \\ \cdots & \vdots & \cdots \\ &\hat{T}_{n}(y_\Aa)& 
                  \end{pmatrix}}_{\Aa \in \Lub}.
\end{align*}
Finally, the mask $\Mmm_{n,p} = (m_{i,j}) \in \Rr^{(n+p) \times (n+1)}$ is given by
\[ m_{i,j} = \left\{ \begin{array}{ll}
                     1, & \text{if}\; (i,j) \in \Gamma_{n,p}, \\
                     1/2, \quad & \text{if}\; (i,j) = (0,n),\\
                     0, & \text{if}\; (i,j) \notin \Gamma_{n,p}^L.
                     \end{array}
\right.\]
For $\Bb \subset \Rr^2$, also the point evaluations $\Lagpol (\Bb)$ of the 
interpolation polynomial can be written compactly as vector-matrix-vector product
\begin{equation} \label{eq:comppol} \Lagpol(\Bb) = \Ttt_{x}(\Bb)^T \Ccc_{n,p} \Ttt_{y}(\Bb) \end{equation}
with $\Ttt_{x}(\Bb)^T = \left(\hat{T}_0(x_\Bb), \cdots, \hat{T}_{n+p-1}(x_\Bb)\right)$ and $\Ttt_{y}(\Bb)^T = \left(\hat{T}_0(y_\Bb), \cdots, \hat{T}_{n}(y_\Bb)\right)$.

\begin{remark}
The formulation of Theorem \ref{thm:interpolation problem} is almost identical to the formulation of Theorem $3$ in \cite{ErbKaethnerAhlborgBuzug2014}. The main structural difference between the two results lies in the form of the index set $\Gamma_{n,p}^L$ and in the weights $w_\Aa$. In the case of non-degenerate Lissajous curves, the index set $\Gamma_{n,p}^L$ is asymptotically $4$ times larger than for degenerate Lissajous curves. Also, in the non-degenerate case there are no vertex points. Structural differences can be also found in the technical aspects of the respective proofs. This is again due to the fact that the operator $\Emb$ maps onto different spaces of trigonometric polynomials in the two cases. 

For the Padua points ($p = 1$), Theorem \ref{thm:interpolation problem} was proven in a different way. The uniqueness of the interpolation problem was derived using an ideal theoretic approach (see \cite{BosDeMarchiVianelloXu2007}), whereas formula \eqref{eq:fundamentalpolynomialsofLagrangeinterpolation} was proven using an explicit formula for the reproducing kernel. 

Also, the matrix formulations in \eqref{eq:compcoeff} and \eqref{eq:comppol} are very similar to the formulations
in \cite{ErbKaethnerAhlborgBuzug2014}. Here, the main difference lies in the form of the mask $\Mmm_{n,p}$ which encodes the index set $\Lubind$. For the Padua points the mask $\Mmm_{n,p}$ is an upper left triangular matrix (cf. \cite{CaliariDeMarchiVianello2008}). 
\end{remark}

\begin{remark}
If we substitute the index set $\Lubind$ in \eqref{eq:index2} by 
\[ \Gamma_{n,p}^{\tilde{L}} := \Lubindbasic \cup \{(n+p,0)\}\]
all the results of this paper can be proven in an analogous way also for the altered 
index set $\Gamma_{n,p}^{\tilde{L}}$. In this way, we obtain in a respectively altered version of Theorem \ref{thm:interpolation problem} the polynomials
\begin{equation*}
 \tilde{L}_{\Aa}(x,y) = w_\Aa \left( K_{n,p}(x,y; x_\Aa,y_\Aa) + \frac12 \hat{T}_{n+p}(x) \hat{T}_{n+p}(x_\Aa) \right), \quad \Aa \in \Lub,
\end{equation*}
as fundamental Lagrange polynomials in the space $\Pi_{n,p}^{2,\tilde{L}} := \spann\{T_{i}(x) T_j(y):\; (i,j) \in \Gamma_{n,p}^{\tilde{L}} \}$. In particular, we can see that it is not possible to speak of a "natural" polynomial space for interpolation on the point set $\Lub$. For the interpolation problem \eqref{eq:interpolationproblem}, several reasonable choices are possible. 
\end{remark}

\section{Convergence results for the interpolation polynomials}

In this section, we study the convergence behavior of the interpolating polynomial $\Lagpol$ to a given function $f \in C([-1,1]^2)$ if $n$ gets large. In particular, for $1 \leq r < \infty$ we prove mean convergence of the Lagrange interpolation in the $L^r$-norms 
\[\|f\|_r^r := \frac1{\pi^2} \int_{-1}^1 \int_{-1}^1  \frac{|f(x,y)|^r}{\sqrt{1-x^2}\sqrt{1-y^2}} \Dx{x} \Dx{y}.\] Further, we will give an upper bound for the growth of the Lebesgue constant.

The proof of the mean convergence as well as the estimate for the Lebesgue constant depend on the following forward quadrature sum estimate. In the upcoming results, we want to keep the parameter $p$ as variable as possible. For this reason, we will take particular care that the constants in the estimates are independent of $p$.

\begin{lemma} \label{lem:forwardquadestimate}
For $1 \leq r < \infty$, the inequality
\begin{equation} \label{eq:forwardquadestimate}
\sum_{\Aa \in \Lub} w_\Aa |P(\Aa)|^r \leq C_r \frac{1}{\pi^2} \int_{-1}^1 \int_{-1}^1  |P(x,y)|^r \frac{1}{\sqrt{1-x^2}} \frac{1}{\sqrt{1-y^2}} \Dx{x} \Dx{y}
\end{equation}
holds for all polynomials $P \in \PiL$ with the constant
$C_r =(r+1)^2 e^2\left(1 + \frac{1}{4\pi} \right)^2$.
\end{lemma}

\begin{proof}
Based on an idea given in \cite{Xu1996}, we use a univariate inequality \cite[Theorem 2]{LubinskyMateNevai1987} to estimate the quadrature sum. Applied to an even trigonometric polynomial $q(t) = \sum_{k=0}^n c_k \cos kt$,
this  inequality has the form
\begin{equation} \label{eq:univariateineq} \sum_{j = 1}^m |q(\theta_j)|^r \leq \left(n + \frac1{2 \eps}\right)  \frac{(r+1)e}{\pi} \int_0^\pi |q(t)|^r \Dx{t},\end{equation}
with $0 \leq \theta_0 < \theta_1 < \cdots < \theta_m \leq 2\pi$ and
$\eps = \min (\theta_1 - \theta_0, \ldots, \theta_m - \theta_{m-1}, 2 \pi - \theta_m + \theta_0) > 0$. 
To estimate the quadrature sum, we use the grid characterization of the set $\Lub$ given in \eqref{def:Luebeckpointsb} and \eqref{def:Luebeckpointsr}. To simplify the calculations,
we will only consider the case where $n$ is odd and $p$ is even, i.e. case (c) in Table \ref{tab:1}. 
The estimates for the cases (a) and (b) in Table \ref{tab:1} follow analogously. We obtain
\begin{align*}
\sum_{\Aa \in \Lub}\!\!\! w_\Aa |P(\Aa)|^r &= \frac{1}{2n(n+p)} \left( \sum_{i = 0}^{\frac{n+p-1}{2}}  \sum_{j = 0}^{\frac{n-1}{2}} 
(2-\delta_{i,0})(2-\delta_{j,0})|P(z_{2i}^{n+p},z_{2j}^{n})|^r 
\right. \\ &\hspace{0.8cm} \left. + \sum_{i = 0}^{\frac{n+p-1}{2}} \sum_{j = 0}^{\frac{n-1}{2}} (2-\delta_{2i+1,n+p})(2-\delta_{2j+1,n})|P(z_{2i+1}^{n+p},z_{2j+1}^{n})|^r \right).
\end{align*}
Now, for every fixed $0 \leq i \leq n+p-1$, we adopt inequality \eqref{eq:univariateineq} to the even trigonometric polynomials
$P(z_{2i}^{n+p},\cos t)$ and $P(z_{2i+1}^{n+p},\cos t)$ of degree $n$. 
We use inequality \eqref{eq:univariateineq} with $\eps = \frac{2\pi}{n}$ and the points $\theta_j = \frac{2 j}{n} \pi$, $j = 0, \ldots, n-1$ for the first summation, as well as
$\eps = \frac{2\pi}{n}$ and the points $\theta_j' = \frac{2 j + 1}{n} \pi$, $j = 0, \ldots, n-1$ for the second summation. We get
\begin{align*}
\sum_{\Aa \in \Lub} \!\!\! w_\Aa |P(\Aa)|^r &\leq \frac{(r+1)e \left(1 + \frac{1}{4\pi} \right)}{2(n+p)\pi} \left( \sum_{i = 0}^{\frac{n+p-1}{2}} \int_{0}^\pi 
(2-\delta_{i,0})|P(z_{2i}^{n+p},\cos t)|^r \Dx{t}
\right. \\ &\hspace{0.8cm} \left. +  \sum_{i = 0}^{\frac{n+p-1}{2}} \int_{0}^\pi (2-\delta_{2i+1,n+p})|P(z_{2i+1}^{n+p},\cos t)|^r \Dx{t} \right).
\end{align*}
Now, on the right hand side, we change the sums with the integrals and apply inequality \eqref{eq:univariateineq} again for the polynomials $P(\cos s, \cos t)$. 
This time, the variable $t$ is fixed and we use $\eps = \frac{2\pi}{n+p}$ with the points $\theta_i = \frac{2 i}{n+p} \pi$ and 
$\theta_i' = \frac{2 i + 1}{n+p} \pi$, $i = 0, \ldots, n+p-1$. In this way, we obtain
\begin{align*}
\sum_{\Aa \in \Lub} \!\!\! w_\Aa |P(\Aa)|^r &\leq (r+1)^2 e^2 \left(1 + \frac{1}{4\pi} \right)^2 \left( \frac1{\pi^2}\int_{0}^{\pi} \int_{0}^\pi 
|P(\cos s,\cos t)|^r \Dx{s} \Dx{t}
\right) \\ &= 
(r+1)^2 e^2 \left(1 + \frac{1}{4\pi} \right)^2  \underbrace{\frac1{\pi^2}\int_{-1}^1 \int_{-1}^1  \frac{|P(x,y)|^r}{\sqrt{1-x^2}\sqrt{1-y^2}} \Dx{x} \Dx{y}}_{= \|P\|_r^r}.
\end{align*} \qed
\end{proof}

As a second auxiliary result, we show the inverse quadrature sum estimate: 

\begin{lemma} \label{lem:inversequadestimate}
For $1 < r < \infty$, the inequality
\begin{equation} \label{eq:inversequadestimate}
\frac{1}{\pi^2} \int_{-1}^1 \int_{-1}^1  |P(x,y)|^r \frac{1}{\sqrt{1-x^2}} \frac{1}{\sqrt{1-y^2}} \Dx{x} \Dx{y} \leq D_r \sum_{\Aa \in \Lub} w_\Aa |P(\Aa)|^r
\end{equation}
holds for all $P \in \PiL$ with a constant
$D_r$ independent of $n$ and $p$. 
\end{lemma}

\begin{proof}
We use a duality argument as described in \cite{Xu1996} (and more generally in \cite{Lubinsky1998}) to show the inverse inequality. For $1 < r < \infty$ and the dual parameter $r' := \frac{r}{r-1}$, we have the representation
\[ \| P \|_r = \sup_{g \in L^{r'}: \|g\|_{r'} = 1} \langle P, g \rangle. \]
Further, if we introduce the partial sums 
\begin{align} S_{n,p} g(x,y) &:= \sum_{(i,j) \in \Lubindbasic} c_{i,j} \hat{T}_i(x) \hat{T}_j(y), \notag\\
S_{n,p}^L g(x,y) &:= \sum_{(i,j) \in \Lubind} c_{i,j} \hat{T}_i(x) \hat{T}_j(y)  \quad c_{i,j} = \langle g, \hat{T}_i(x) \hat{T}_j(y) \rangle, \label{eq:partialsumLub}
\end{align}
we obtain for polynomials $P \in \PiL$:
\begin{align*} 
\| P \|_r &= \sup_{g \in L^{r'}: \|g\|_{r'} = 1} \langle P, S_{n,p}^L g \rangle \\ &\leq \sup_{g \in L^{r'}: \|g\|_{r'} = 1} \langle P, S_{n,p} g \rangle + \sup_{g \in L^{r'}: \|g\|_{r'} = 1} c_{0,n} \langle P, \hat{T}_{n}(y) \rangle.
\end{align*}
Now, we use the fact that $\mathcal{L}_{n,p} P = P$ holds for all polynomials $P \in \PiL$ and the representation \eqref{eq:coeffinterpolation} of $\mathcal{L}_{n,p} P$ in the Chebyshev basis. In this way, 
we obtain by the orthonormality of the basis $\hat{T}_i(x) \hat{T}_j(y)$:
\begin{align*}
\| P \|_r &\leq  \sup_{\|g\|_{r'} = 1} \!\!\! \sum_{\quad \Aa \in \Lub}\!\!\! w_\Aa P(\Aa) \, S_{n,p} g(\Aa) + \sup_{ \|g\|_{r'} = 1} \frac{c_{0,n}}{2}   \!\!\! \sum_{\;\;\;\Aa \in \Lub} \!\!\! w_\Aa P(\Aa) \, \hat{T}_{n}(y_\Aa).
\end{align*}
Applying H\"olders inequality for the dual pair $r,r'$ to both sums on the right hand side, we get
\begin{align*}
\| P \|_r &\leq  \left( \sum_{\Aa \in \Lub} \!\!\!  w_\Aa |P(\Aa)|^r \right)^{\frac1r} \sup_{ \|g\|_{r'} = 1} \left( \sum_{\Aa \in \Lub} \!\!\!  w_\Aa |S_{n,p}g(\Aa)|^{r'} \right)^{\frac1{r'}} \\ & \quad + \left( \sum_{\Aa \in \Lub} \!\!\!\!\!\!  w_\Aa |P(\Aa)|^r \right)^{\frac1r} \sup_{ \|g\|_{r'} = 1} \frac{c_{0,n}}{2} \left( \sum_{\Aa \in \Lub} \!\!\!  w_\Aa |\hat{T}_{n}(y_\Aa)|^{r'} \right)^{\frac1{r'}}.
\end{align*}
Now, for the first term in the above inequality, we use the forward quadrature sum estimate proven in Lemma \ref{lem:forwardquadestimate}. For the second term, we use the fact that $\hat{T}_n(y) \leq \sqrt{2}$ and, in a second step, again the H\"older inequality for $\langle g, \hat{T}_n(y) \rangle$. Then, we obtain
\begin{align*}
\| P \|_r &\leq  \left( \sum_{\Aa \in \Lub} \!\!\!  w_\Aa |P(\Aa)|^r \right)^{\frac1r} \left( C_{r'}^{\frac1{r'}} \sup_{ \|g\|_{r'} = 1} \| S_{n,p} g\|_{r'} + \sup_{ \|g\|_{r'} = 1} \frac{\langle g, \hat{T}_n(y) \rangle }{\sqrt{2}}\right) \\
& \leq \left( \sum_{\Aa \in \Lub} \!\!\!  w_\Aa |P(\Aa)|^r \right)^{\frac1r} \left(C_{r'}^{\frac1{r'}} \| S_{n,p}\|_{L^{r'} \to L^{r'}} + 1 \right).
\end{align*}
Therefore, if we know that the partial sum operators $S_{n,p}$ are uniformly bounded in the $L_r$ norm, the proof is finished. To see this, we write $S_{n,p} g $ as trigonometric partial sum
\begin{align} S_{n,p} g(\cos \alpha, \cos \beta) &= 
\sum_{(i,j) \in \Gamma_{n,p}^{\mathrm{trig}}}\!\!\! a_{i,j}\, e^{\mathfrak{i}(\alpha i + \beta j)}, \notag \\ \Gamma_{n,p}^{\mathrm{trig}} &:= \left\{(i,j) \in \Zz^2: \, \frac{|i|}{n+p}+\frac{|j|}{n} < 1 \right\}. \label{eq:rhomboid}
\end{align}
The index set $\Gamma_{n,p}^{\mathrm{trig}}$ is the intersection of $\Zz^2$ with a rhombus consisting of four congruent right triangles with side lengths $n+p$ and $n$. Now, we can adopt a classical result of Fefferman~\cite{Fefferman1971}. It states that for $1 < r' < \infty$ the trigonometric partial sum over 
the rhombic index set $\Gamma_{n,p}^{\mathrm{trig}}$ is uniformly bounded in the $L^{r'}$-norm on the $2$-torus, i.e.
\[ \|S_{n,p} g\|_{r'} = \left(\frac{1}{4 \pi^2} \int_{[0,2\pi)^2} \Big|\sum_{(i,j) \in \Gamma_{n,p}^{\mathrm{trig}}}\!\!\! a_{i,j}\, e^{\mathfrak{i}(\alpha i + \beta j)}
\Big|^{r'} \Dx{\alpha}\Dx{\beta}\right)^{\frac1{r'}} \leq C_{r'}^F \|g\|_{r'}, \]
holds for all $n,p \in \Nn$ and the constant $C_{r'}^F$ does not depend on $n$ and $p$. This proves the statement of the Lemma with $D_r = (C_{r'}^{\frac1{r'}} C_{r'}^F +1)^r$. \qed
\end{proof}

\begin{remark}
The combination of forward \eqref{eq:forwardquadestimate} and inverse quadrature sum estimate \eqref{eq:inversequadestimate} is also referred to as Marcinkiewicz-Zygmund inequality, see \cite{Lubinsky1998,MhaskarPrestin1998} and the references therein. The idea for the proofs of Lemma \ref{lem:forwardquadestimate} and Lemma \ref{lem:inversequadestimate} is taken from \cite{Xu1996} where similar results are shown for the Xu points. 
The detailed elaboration of the proofs of the inequalities
\eqref{eq:forwardquadestimate} and \eqref{eq:inversequadestimate} was necessary in order to guarantee that the constants in \eqref{eq:forwardquadestimate} and \eqref{eq:inversequadestimate} do not depend on the parameter $p$.  
\end{remark}

\begin{theorem} \label{thm-meanconvergence}
Let $1 \leq r < \infty$, $f \in C([-1,1]^2)$ and $(p_n)_{n \in \Nn}$ a sequence of natural numbers such that $n$ and $n+p_n$ are relatively prime for all $n \in \Nn$. Then, the Lagrange interpolant $\mathcal{L}_{n,p_n} f $ converges for $n \to \infty$ to the function $f$ in the $L^r$-norm, i.e.
\[ \lim_{n \to \infty} \| \mathcal{L}_{n,p_n} f  - f\|_r = 0. \] 
\end{theorem}

\begin{proof}
Following the argumentation scheme described in \cite{Lubinsky1998}, we adopt the inverse quadrature sum estimate of Lemma \ref{lem:inversequadestimate} to the polynomial $\mathcal{L}_{n,p_n} f \in \PiL$. In this way, we get for $1 < r < \infty$:
\[ \| \mathcal{L}_{n,p_n} f \|_r^r \leq D_r \sum_{\Aa \in \Gamma_{n,p_n}^L} w_\Aa |f(\Aa)|^r \leq D_r \|f\|_\infty^r. \] 
Now, for an arbitrary polynomial $P \in \Pin$, we have the identity $\mathcal{L}_{n,p_n} P = P$ and therefore
\[ \|\mathcal{L}_{n,p_n}f - f\|_r \leq \|\mathcal{L}_{n,p_n}(f - P)\|_r + \|P-f\|_r \leq (1+D_r)\|P-f\|_\infty.\] 
Since polynomials are dense in $C([-1,1]^2)$, we immediately get mean convergence of the Lagrange interpolant for $1 < r < \infty$. The convergence for $r = 1$ follows analogously using the estimate
\[ \|\mathcal{L}_{n,p_n} f \|_1 \leq \|\mathcal{L}_{n,p_n} f \|_2   \leq D_2 \|f\|_\infty. \] \qed
\end{proof}

Finally, we consider the Lebesgue constant related to the interpolation problem \eqref{eq:interpolationproblem}. It is given as the operator norm of the interpolation operator $\mathcal{L}_{n,p}$ in the space $(C([-1,1]^2), \| \cdot \|_\infty)$:
\[ \Lambda_{n,p} := \max_{(x,y) \in [-1,1]^2} \sum_{\Aa \in \Lub} |L_{\Aa}(x,y)|.\]
For $p = 1$, it is known that $\Lambda_{n,1}$ grows as $\mathcal{O}(\ln^2 n)$~\cite{BosDeMarchiVianelloXu2006,VecchiaMastroianniVertesi2009}. The next theorem states that a similar behavior is true for general $p$. 

\begin{theorem} \label{thm:lebesgueconstant}
The Lebesgue constant $\Lambda_{n,p} $ is bounded by
\[ D_{\Lambda} \ln^2(n) \leq \Lambda_{n,p} \leq C_{\Lambda} \ln^2(n+p).\] 
The constants $C_{\Lambda}$ and $D_{\Lambda}$ do not depend on $n$ and $p$. 
\end{theorem}

\begin{proof}
Using formula \eqref{eq:fundamentalpolynomialsofLagrangeinterpolation}, we get for the Lebesgue constant $\Lambda_{n,p}$:
\begin{align*} \Lambda_{n,p} &= \max_{(x,y) \in [-1,1]^2} \sum_{\Aa \in \Lub} w_\Aa \left| K_{n,p}(x,y; x_\Aa,y_\Aa) + \frac12 \hat{T}_{n}(y) \hat{T}_{n}(y_\Aa) \right| \\ & \leq
\max_{(x,y) \in [-1,1]^2} \sum_{\Aa \in \Lub} w_\Aa \left| K_{n,p}(x,y; x_\Aa,y_\Aa) \right| +1.
\end{align*}
Now, by Lemma \ref{lem:forwardquadestimate}, we get
\begin{align*} \Lambda_{n,p} & \leq C_1
\max_{(x,y) \in [-1,1]^2} \| K_{n,p}(x,y; \cdot, \cdot) \|_1 +1.
\end{align*}
With the coordinate transforms $x = \cos \alpha$ and $y = \cos \beta$ we transfer the above norm in a trigonometric setting on the $2$-torus and obtain
\begin{align*} \Lambda_{n,p} & \leq 
\max_{0 \leq \alpha , \beta < 2\pi} \frac{C_1}{4 \pi^2} 
\underset{[0,2\pi)^2}{\int} \left| \underset{(i,j) \in \Gamma_{n,p}^{\mathrm{trig}}}{\sum} e^{\mathfrak{i} ( i \alpha' +j \beta')}
\cos (i \alpha) \cos (j \beta)\right| 
\Dx{\alpha'} \Dx{\beta'}  +  1 \\
& = C_1 \frac{1}{4 \pi^2} \int_{[0,2\pi)^2} \left|\underset{(i,j) \in \Gamma_{n,p}^{\mathrm{trig}}}{\sum}  e^{\mathfrak{i} ( i t\alpha'+ j \beta')} \right|\Dx{\alpha'} \Dx{\beta'} + 1,
\end{align*}
where $\Gamma_{n,p}^{\mathrm{trig}}$ denotes the rhombic index set defined in \eqref{eq:rhomboid}. The integral on the right hand side 
corresponds to the $L^1$-norm of the Dirichlet kernel with respect to the rhombic summation area $\Gamma_{n,p}^{\mathrm{trig}}$. By a result of \cite{YudinYudin1985}, this norm is bounded by $C \ln^2(2n+2p)$ with a constant $C$ independent of $n$ and $p$. Thus, we get
\[\Lambda_{n,p} \leq C_1 C \ln^2(2n+2p) + 1 \leq (4 C_1 C+1) \ln^2(n+p)\]
and the upper estimate is proven for $C_\Lambda = 4 C_1 C + 1$.

For the lower estimate of the Lebesgue constant, we proceed similar as in a proof given for the Padua points ~\cite{VecchiaMastroianniVertesi2009}. Since $\Lagpol$ is a linear projection of $C([-1,1]^2)$ onto $\PiL$, we get by \cite[Theorem 2.3]{SziliVertesi2009154} the following estimate:
\[ \Lambda_{n,p} = \sup_{\|f\|_\infty \leq 1} \|\Lagpol\|_\infty
\geq \frac12 \sup_{\|f\|_\infty \leq 1} \| S_{n,p}^L f\|_\infty,\]
where $S_{n,p}^L f = \langle f, K_{n,p}^L(x,y; \cdot, \cdot)\rangle$ denotes the partial sum operator given in 
\eqref{eq:partialsumLub}. We note that in \cite{SziliVertesi2009154} this inequality is only proven for projections onto $\Pin = \Pi_{n,1}^{2,L}$. However, a straightforward modification of the proof in \cite{SziliVertesi2009154} yields the respective result for general $p$. As a consequence of the Riesz representation theorem (cf. \cite[IV.6.3]{DunfordSchwartz1}) we further obtain
\begin{align*} \Lambda_{n,p} & \geq \frac12 \max_{(x,y) \in [-1,1]^2} 
\left\| K_{n,p}^L(x,y;\cdot,\cdot) \right\|_1 \geq \frac12 \left\| K_{n,p}^L(1,1;\cdot,\cdot) \right\|_1.
\end{align*}
Using the coordinate transform $x = \cos \alpha$, $y = \cos \beta$, we transfer the above norm in the trigonometric setting
\begin{align*} \Lambda_{n,p} & \geq \frac12 \frac1{4\pi^2} \int_{[0,2\pi)^2} \left| \underset{(i,j) \in \Gamma_{n,p}^{L, \mathrm{trig}}}{\sum} e^{\mathfrak{i} ( i \alpha+ j \beta)} \right| \Dx{\alpha} \Dx{\beta}
\end{align*}
with the index set $\Gamma_{n,p}^{L, \mathrm{trig}} = \Gamma_{n,p}^{ \mathrm{trig}} \cup (0, \pm n)$. 
Thus, on the right hand side we have again the $L^1$-norm of a Dirichlet kernel based on a rhombic summation index $\Gamma_{n,p}^{L, \mathrm{trig}}$. Since a sphere with radius $\frac{n}{\sqrt{2}}$ fits into the rhombus with diagonal lengths $2n$ and $2n+p$ we can adopt a result of \cite{Yudin1979}. It states that the $L^1$-norm of the Dirichlet under investigation is then bounded from below by $ D \ln^2(\frac{n}{\sqrt{2}}+1)$ with a constant $D$ independent of $n$ and $p$. Therefore, we get for $D_\Lambda = \frac{D}{4}$:
\[  \Lambda_{n,p} \geq \frac{D}{2} \ln^2\left(\frac{n}{\sqrt{2}}+1\right) \geq \frac{D}{4} \ln^2(n) = D_\Lambda \ln^2(n).\] \qed
\end{proof}

For a function $f \in C([-1,1]^2)$, the best approximation $E_{n,p}(f)$ of $f$ in the polynomial space $\PiL$ is given as
\[ E_{n,p}(f) := \min_{P \in \PiL} \|f - P\|_\infty.\]
For $u,v \geq 0$, the modulus of continuity $\omega(f;u, v)$ is defined as (cf. \cite[section 3.4.1]{Timan1960})
\[\omega(f;u, v) := \sup_{\substack{|x_1-x_2| \leq u \\ x_1,x_2 \in [-1,1]}} \sup_{\substack{|y_1-y_2| \leq v \\ y_1,y_2 \in [-1,1]}} |f(x_1,y_1)-f(x_2,y_2)|. \]
Using these standard tools from constructive approximation theory together with the estimate of the Lebesgue constant in Theorem \ref{thm:lebesgueconstant}, we
obtain the following error estimates. 

\begin{corollary} \label{cor-dinilipschitz}
For any continuous function $f \in C([-1,1]^2)$, we have
\begin{equation} \label{eq:bestapproximation}
\|f - \Lagpol \|_\infty \leq (C_{\Lambda} \ln^2(n+p)+1) E_{n,p}(f).
\end{equation}
If $\frac{\partial^r f}{\partial x^r}, \frac{\partial^s f}{\partial y^s} \in C([-1,1]^2)$ for given $r, s \in \Nn_0$, we further have the estimate 
\begin{equation} \label{eq:Dini1} \|f - \Lagpol \|_\infty \leq  C 
\ln^2(n+p) \left( \frac{\omega\left( \frac{\partial^r f}{\partial x^r}; \frac{1}{n+p}, 0 \right)}{(n+p)^r}  
+ \frac{\omega \left(  \frac{\partial^s f}{\partial y^s}; 0,\frac{1}{n} \right)}{n^s}  \right) \end{equation}
with a constant $C$ independent of $n$ and $p$. 
\end{corollary}

\begin{proof}
We denote by $P^*$ the best approximating polynomial of $f$ in $\PiL$, i.e. $\|P^* - f \|_\infty = E_{n,p}(f)$. Since $\mathcal{L}_{n,p} P^* = P^*$, 
the estimate of the Lebesgue constant in Theorem \ref{thm:lebesgueconstant} leads to the following estimate:
\begin{align*} \|f - \Lagpol \|_\infty &\leq \|f - P^* \|_\infty + \| \mathcal{L}_{n,p} P^* - \Lagpol \|_\infty \\ & \leq (\Lambda_{n,p} + 1) \|f - P^* \|_\infty  
= (C_{\Lambda}  \ln^2(n+p)+1) E_{n,p}(f).
\end{align*}
Since the polynomial space $\Pi_R^2 = \spann \left\{T_i(x) T_j(y): \; 0 \leq i < \frac{n+p}{2}, 0 \leq j < \frac{n}{2} \right\}$ is contained in $\PiL$, we have 
$E_{n,p}(f) \leq \|f-Q^*\|_\infty$ where $Q^*$ denotes the best approximating polynomial in $\Pi_R^2$. 
Now, a multivariate version of Jackson's inequality (cf. \cite[section 5.3.2]{Timan1960}) gives
\[ \|f-Q^*\|_\infty \leq \textstyle \tilde{C} \left(  \frac{2^{r} \omega\left( \frac{\partial^r f}{\partial x^r}; \frac{2}{n+p}, 0 \right)}{(n+p)^r}  
+ \frac{2^{s} \omega \left(  \frac{\partial^s f}{\partial y^s}; 0,\frac{2}{n} \right)}{n^s}  \right) \leq \tilde{C} 2^{\max \{r,s\} +1} 
\left( \frac{\omega\left( \frac{\partial^r f}{\partial x^r}; \frac{1}{n+p}, 0 \right)}{(n+p)^r}  
+ \frac{\omega \left(  \frac{\partial^s f}{\partial y^s}; 0,\frac{1}{n} \right)}{n^s}  \right).\]
In the second inequality, we used the semi-additivity of the modulus $\omega(f;h_x, h_y)$. This inequality together with \eqref{eq:bestapproximation}
yields \eqref{eq:Dini1}.  \qed
\end{proof}

\begin{remark}
If the function $f \in C([-1,1]^2)$ satisfies the Dini-Lipschitz-type condition 
\[ \lim_{n \to \infty}\ln^2(n+p_n) \, \omega \left(f; \frac{1}{n+p_n}, \frac{1}{n}\right) = 0, \] 
inequality \eqref{eq:Dini1} guarantees the uniform convergence 
\[ \lim_{n \to \infty} \| \mathcal{L}_{n,p_n} f - f \|_\infty = 0. \]
For the Padua points, the result of Corollary \ref{cor-dinilipschitz} can be found in \cite{CaliariDeMarchiVianello2008}. The upper estimate for the Lebesgue constant of the Padua points proven in \cite{BosDeMarchiVianelloXu2006} 
is more accurate compared to the estimate in Theorem \ref{thm:lebesgueconstant}. 
Further, more recent versions of Jackson's inequality in a general multidimensional setting can be found in \cite{BagbyBosLevenberg2002} and the references therein.
\end{remark}

\begin{remark}
With slight adaptations of the respective proofs, 
the convergence results in Theorem \ref{thm-meanconvergence} and Corollary \ref{cor-dinilipschitz} 
as well as the estimate of the Lebesgue constant in Theorem \ref{thm:lebesgueconstant} can be shown also
for the interpolation schemes of the non-degenerate Lissajous curves considered in \cite{ErbKaethnerAhlborgBuzug2014}.
In particular, these results confirm the numerical tests given in \cite{ErbKaethnerAhlborgBuzug2014}. 
\end{remark}

\section{Numerical experiments}

Finally, we illustrate numerically the effect of different values of the parameter $p_n$, $n \in \Nn$, on the Lebesgue constant and the convergence behavior of 
the polynomial interpolation schemes. In particular, we will see that larger values of $p_n$ can have advantages when approximating functions on anisotropic domains or 
functions with anisotropic smoothness. 

We investigate first the Lebesgue constant $\Lambda_{n,p_n}$ for the three different parameters $p_n \in \left\{1, n+1, \lfloor \sqrt{n} \rfloor n+1 \right\}$. 
The first choice $p_n = 1$ leads to respective results of the Padua points and can be compared to the numerical experiments given in 
\cite{CaliariDeMarchiVianello2005,ErbKaethnerAhlborgBuzug2014}. 
In Figure~\ref{fig:lebesgue}(a) the values $\Lambda_{n,p_n}$ are illustrated for $1 \leq n \leq 50$. For a better comparison of the values,
also the functions $f_1(n) = \ln(n)^2/2 + 2$ and $f_2(n) = \ln(n \sqrt{n})^2/2 + 4$ are plotted in Figure~\ref{fig:lebesgue}(a) as a lower and an upper benchmark, respectively. In Figure~\ref{fig:lebesgue}(b)
the Lebesgue constants are plotted with respect to the number $|\mathrm{LD}_{n,p_n}|$ of interpolation points. 
The logarithmic growth of $\Lambda_{n,p_n}$ as estimated in Theorem \ref{thm:lebesgueconstant} is clearly visible in Figure~\ref{fig:lebesgue}. Further, the numerical
experiments indicate a slight growth of $\Lambda_{n,p_n}$ with respect to an increasing parameter $p_n$. The best results are obtained for $p_n = 1$, i.e. for the Padua points.

\begin{figure}[htb]
	\centering
	\subfigure[ $\Lambda_{n,p_n}$ for $n = 1, \ldots 50$. ]{\includegraphics[scale=0.42]{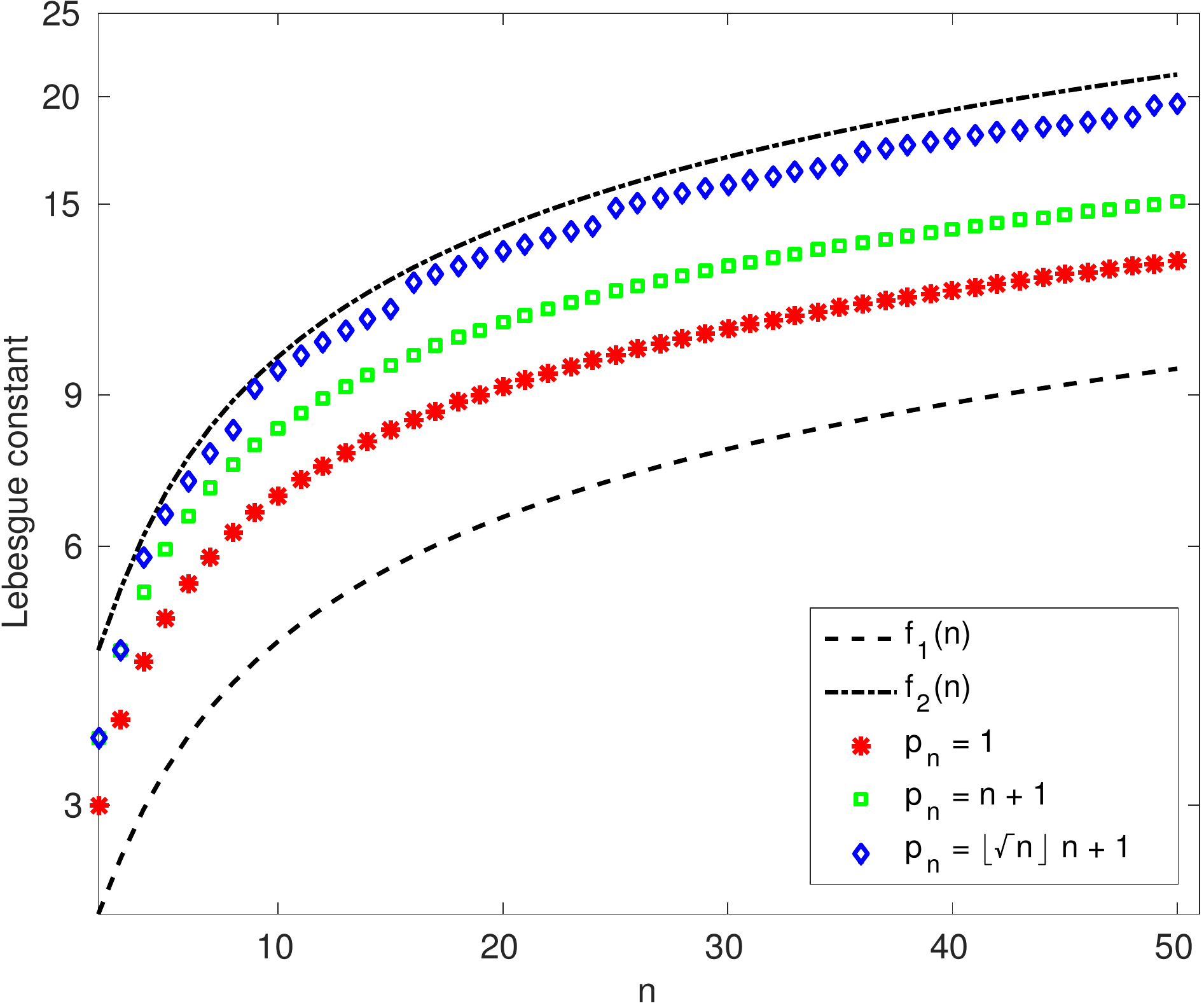}}
	\hfill	
	\subfigure[ $\Lambda_{n,p_n}$ with respect to the number $|\mathrm{LD}_{n,p_n}|$.]{\includegraphics[scale=0.42]{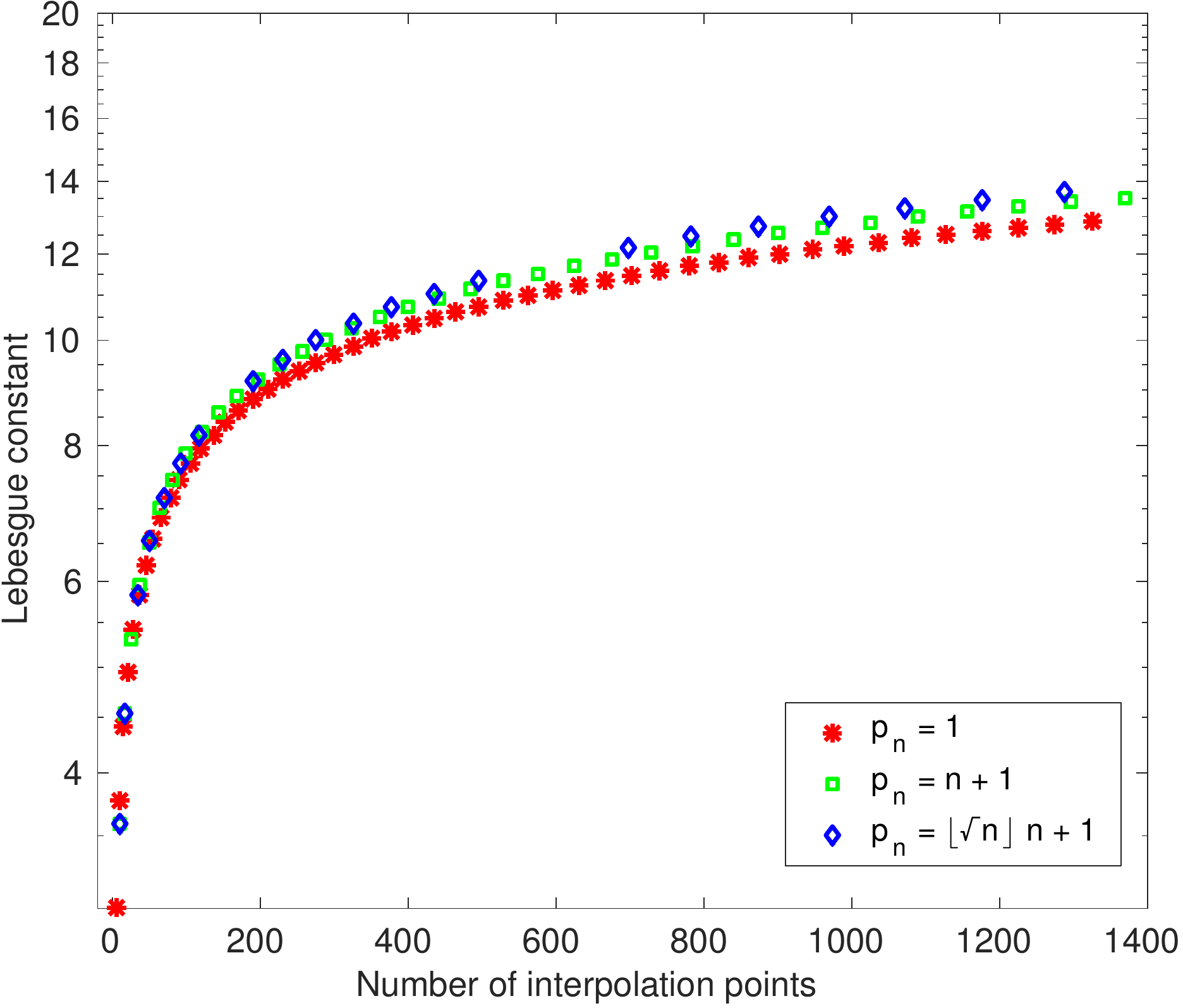}} 
  	\caption{The Lebesgue constant $\Lambda_{n,p_n}$ for the parameters $p_n \in \left\{1, n+1, \lfloor \sqrt{n} \rfloor n+1 \right\}$.}
	\label{fig:lebesgue}
\end{figure}

To evaluate the convergence of the interpolation polynomials $\mathcal{L}_{n,p_n} f$ to a continuous function $f$, we use the three test functions 
\begin{align*}
f_1(x,y) &= e^{-\frac{(5-10x)^2}{2}}+0.75 e^{-\frac{(5-10y)^2}{2}}+ 0.75 e^{-\frac{(5-10x)^2+(5-10y)^2}{2}}, \\
f_{2}(x,y) &= \left(\1_{[0.25,0.75]}(x) (x-0.25) + \1_{[0.75, \infty)}(x) \right) e^{-\frac{(y-0.5)^2}{2}} , \\
f_3(x,y) &= f_2(y,x),
\end{align*}
where $\1_J(x)$ denotes the indicator function of an interval $J \subset \Rr$, i.e. $\1_J(x) = 1$ if $x \in J$ and $\1_J(x) = 0$ otherwise. 
The function $f_1$ is taken from the test set in \cite{RenkaBrown1999} and is smooth, whereas $f_{2}$ and $f_3$ have two discontinuities 
in the partial derivatives $\frac{\partial f_{2}}{\partial x}$ and $\frac{\partial f_{3}}{\partial y}$, respectively. 

As described in \cite{CaliarideMarchiVianello2008-2}, we use affine mappings of the square $[-1,1]^2$
to calculate the interpolation polynomials $\mathcal{L}_{n,p_n} f$ on the two rectangular domains $\Omega_1 = [0,1]^2$ and $\Omega_2 = [0,2] \times [0,1]$. As 
parameters $p_n$ we consider again the cases $p_n \in \left\{1, n+1, \lfloor \sqrt{n} \rfloor n+1 \right\}$. The maximal error between $f$ and $\mathcal{L}_{n,p_n}f$ 
is computed on a uniform grid of $100\times100$ and $200\times100$ points defined in $\Omega_1$ and $\Omega_2$, respectively. 

\begin{figure}[htb]
	\centering
	\subfigure[ $ \| \mathcal{L}_{n,p_n}f_{1} - f_{1} \|_{\infty}$ on $\Omega_1$. ]{\includegraphics[scale=0.42]{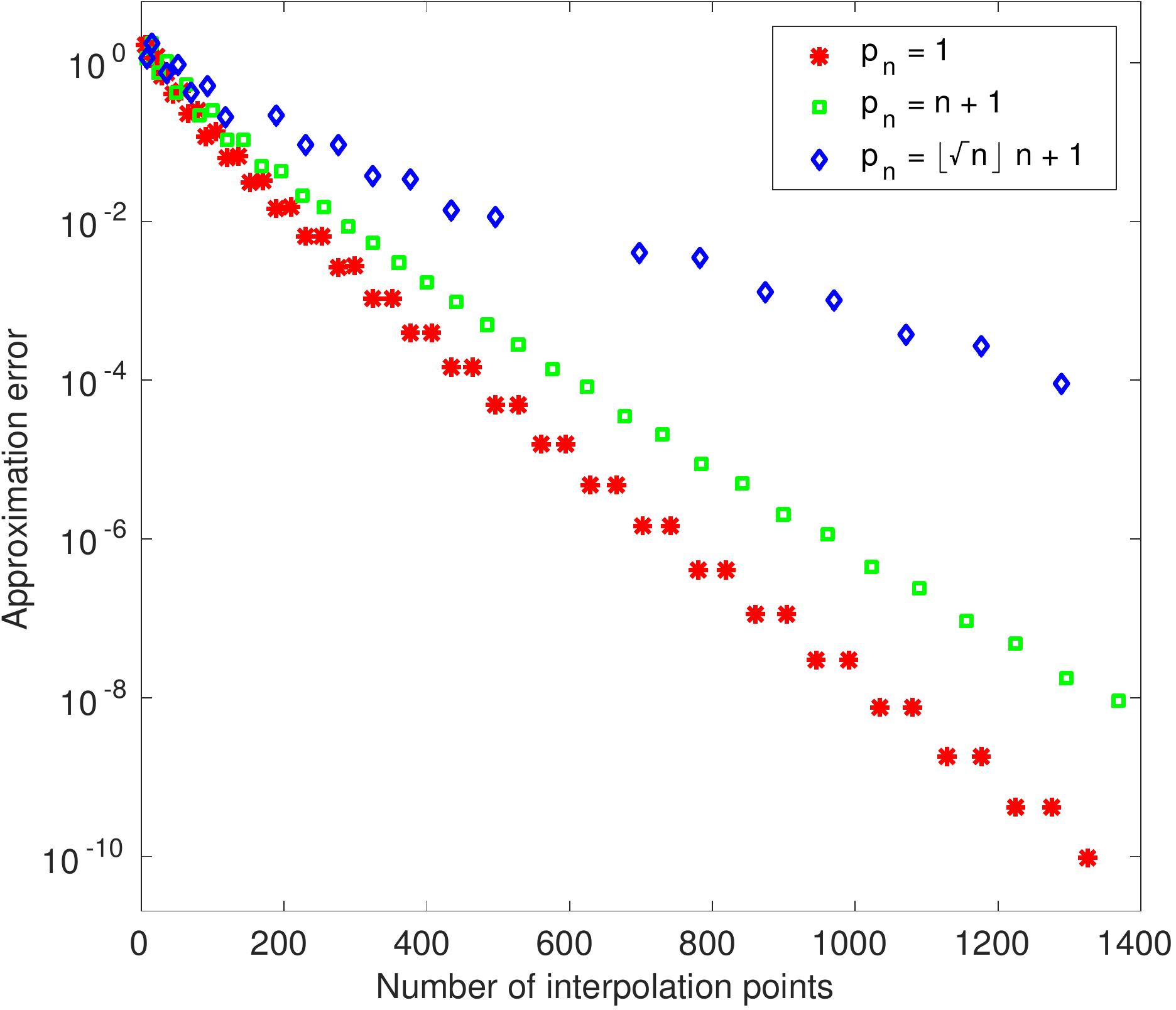}}
	\hfill	
	\subfigure[ $ \| \mathcal{L}_{n,p_n}f_{1} - f_{1} \|_{\infty}$ on $\Omega_2$.]{\includegraphics[scale=0.42]{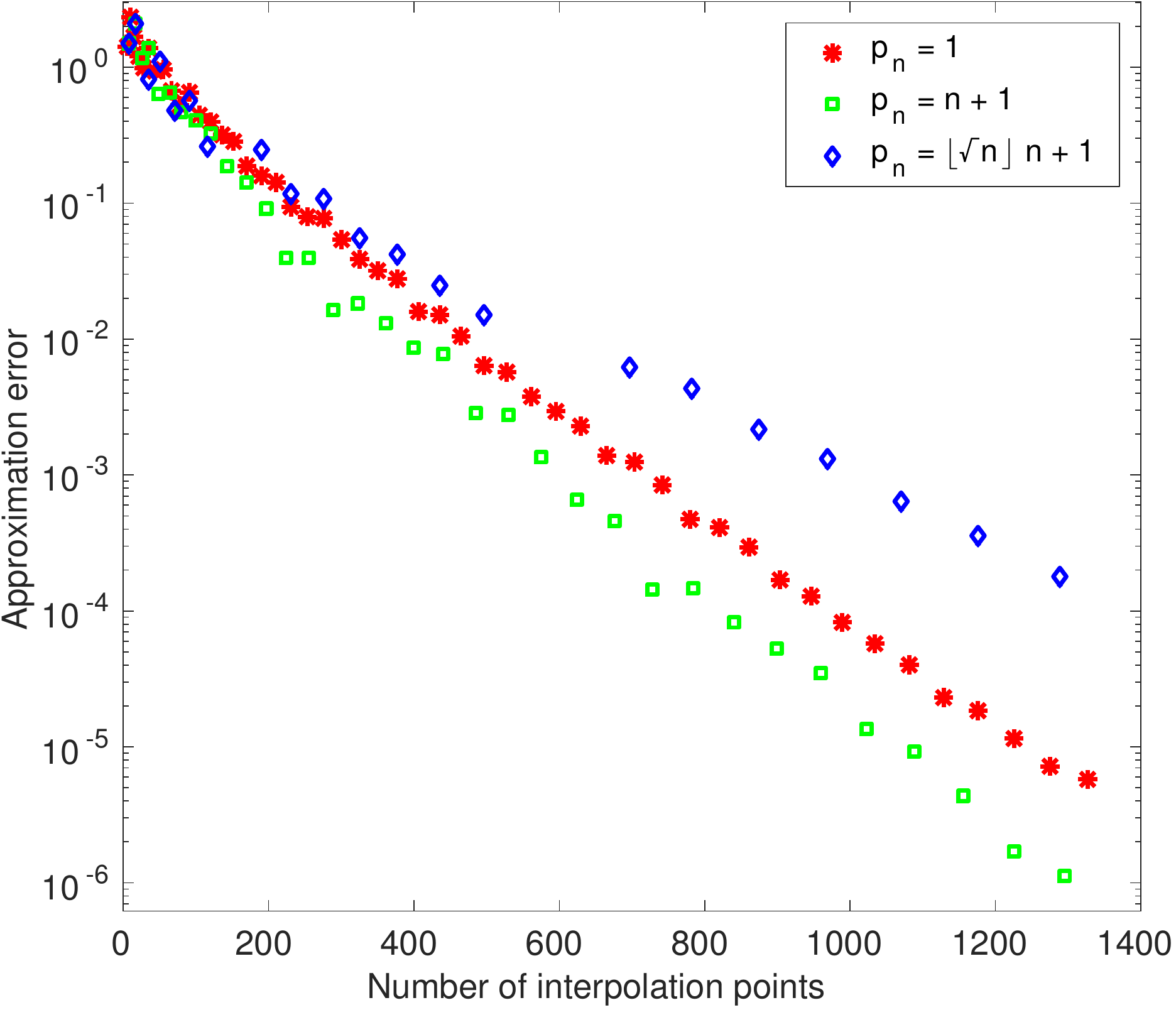}} 
	\caption{The approximation error $\| \mathcal{L}_{n,p_n}f_{1} - f_{1} \|_{\infty}$ on the domains $\Omega_1 = [0,1]^2$ and $\Omega_2 = [0,2] \times [0,1]$.}
	\label{fig:appr1}
\end{figure}

The approximation errors for the function $f_1$ on the domains $\Omega_1$ and $\Omega_2$ are displayed in Figure~\ref{fig:appr1}(a) and \ref{fig:appr1}(b), respectively. 
While all considered interpolation polynomials converge to $f_1$ as $n \to \infty$, the decay of the approximation error on the two domains depends on the choice of 
the parameter $p_n$. On the square $\Omega_1$ the best results are obtained for the parameter $p_n = 1$, whereas the parameter $p_n = n + 1$ seems to produce better adapted
interpolation polynomials for the anisotropic domain $\Omega_2$. 

Finally, we investigate the influence of $p_n$ in the case that the given function 
is smooth in only one of the variables. As test functions, we consider $f_2$ and $f_3$ on $\Omega_1$. Since both are Lipschitz on $\Omega_1$, the criterion in
Remark 7 is satisfied for all three choices of $p_n$. The convergence of the respective interpolation schemes is visible in Figure~\ref{fig:appr2}.
We observe that for $f_2$ the choice $p_n = \lfloor \sqrt{n} \rfloor n+1$ produces the best result, 
whereas for the function $f_3$ the choice $p_n = 1$ 
leads to a faster decay with respect to the number of interpolation points. 

This numerical result reflects the
theoretical estimate given in \eqref{eq:Dini1} which contains an interplay between two directional moduli of continuity. 
Since $f_2$ is smooth with respect to $y$ and nonsmooth with respect to $x$, 
the modulus $\omega(f_2;0,\frac{1}{n})$ is typically smaller than the modulus $\omega(f_2;\frac{1}{n+p_n},0)$ if $p_n$ is small. 
However, using large values for the parameter $p_n$, this disparity between the two directional moduli can be compensated and 
a better approximation quality can be achieved. For the function
$f_3$ the roles of $x$ and $y$ are interchanged. In this case, larger values of $p_n$ have no essential effect on the approximation quality.

\begin{figure}[htb]
	\centering
	\subfigure[ $ \| \mathcal{L}_{n,p_n}f_{2} - f_{2} \|_{\infty}$ on $\Omega_1$. ]{\includegraphics[scale=0.42]{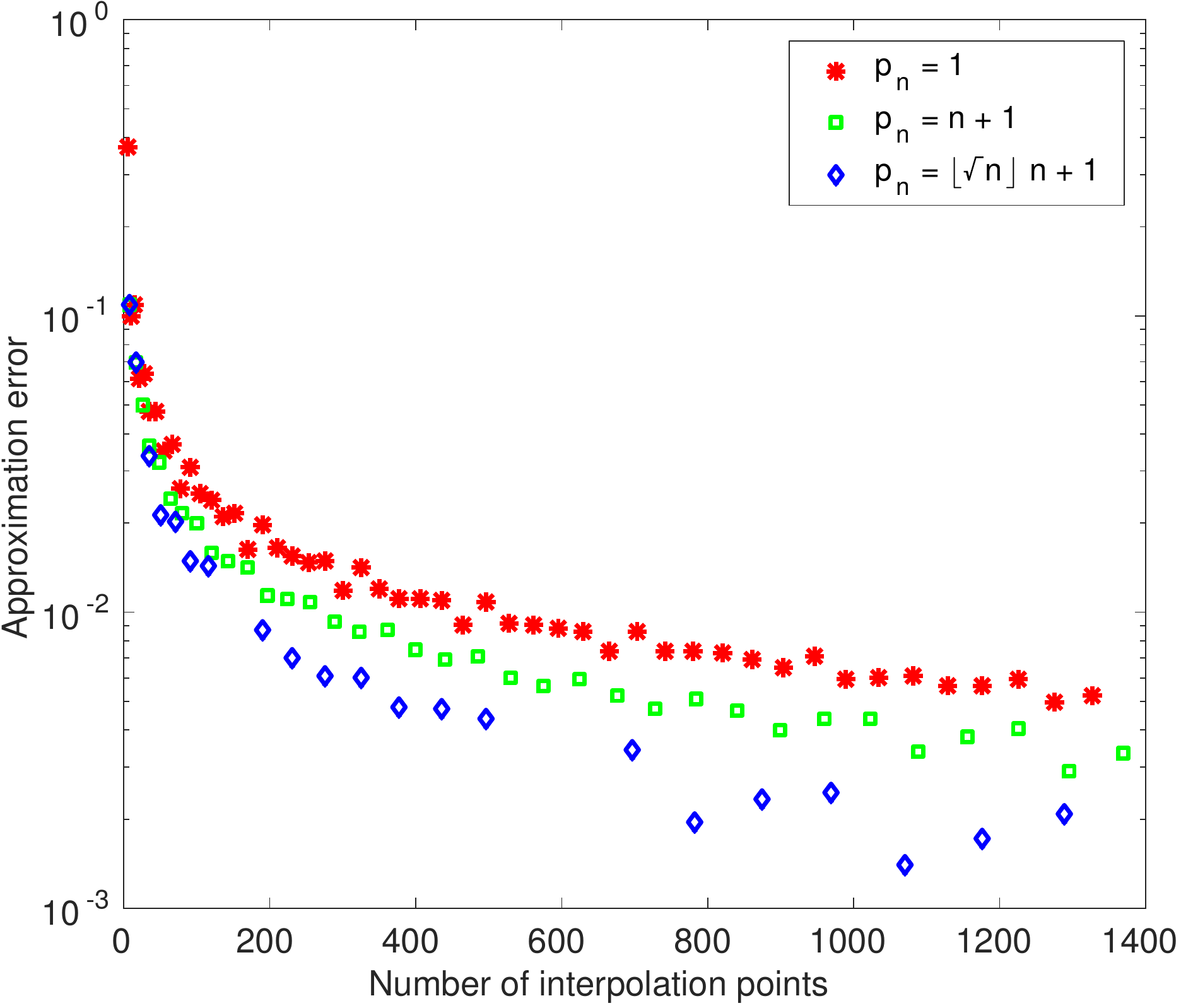}}
	\hfill	
	\subfigure[ $ \| \mathcal{L}_{n,p_n}f_{3} - f_{3} \|_{\infty}$ on $\Omega_1$.]{\includegraphics[scale=0.42]{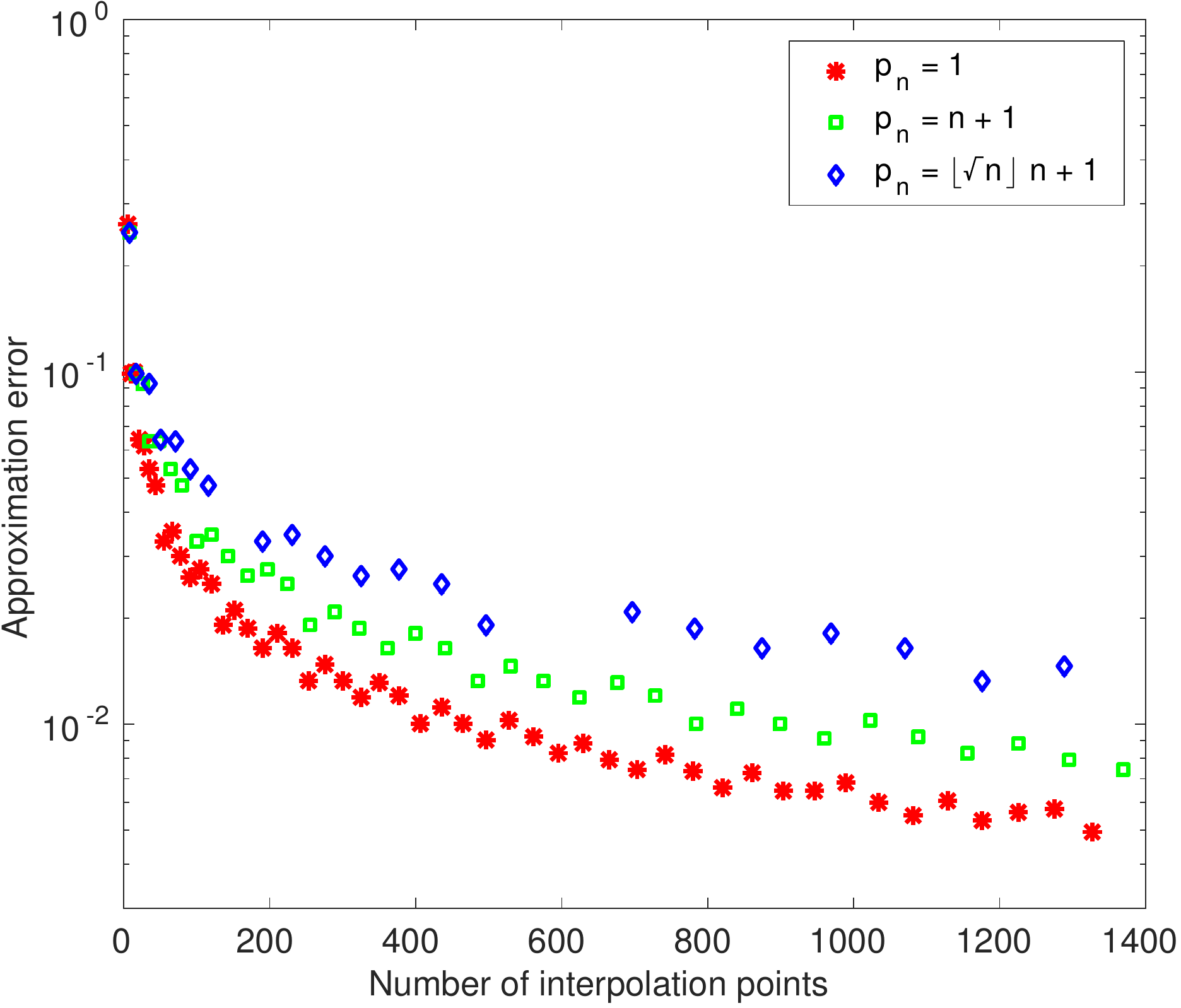}} 
	\caption{The approximation errors $\| \mathcal{L}_{n,p_n}f_{2} - f_{2} \|_{\infty}$ and 
	$\| \mathcal{L}_{n,p_n}f_{3} - f_{3} \|_{\infty}$ on the domain $\Omega_1 = [0,1]^2$.}
	\label{fig:appr2}
\end{figure}

\section*{Acknowledgements}
The author gratefully acknowledges the financial support of the German Research Foundation (DFG, grant number ER 777/1-1).
Further, he thanks an anonymous referee for a lot of valuable comments that improved the quality of the article.

\end{document}